\documentclass[12pt]{amsart}
\usepackage{amsmath, amsthm, amscd, amssymb, amsfonts, latexsym, color}
\usepackage{fullpage}
\usepackage{bm}
\usepackage{graphicx}
\usepackage[all]{xy}

\newcommand{\F}{\mathbb F}

\newcommand{\Z}{\mathbb Z}
\newcommand{\Q}{\mathbb Q}
\newcommand{\R}{\mathbb R}
\newcommand{\co}{\mathcal O}
\newcommand{\C}{\mathbb C}

\newcommand{\CCD}{\mathcal {D}}

\newcommand{\CS}{\mathcal {V}}

\newcommand{\CF}{\mathcal {F}}
\newcommand{\od}{\mathfrak{n}}

\newcommand{\PP}{\mathbb P}

\DeclareMathOperator{\Gal}{Gal}
\DeclareMathOperator{\Aut}{Aut}
\DeclareMathOperator{\Res}{Res}
\DeclareMathOperator{\Prob}{Prob}
\DeclareMathOperator{\re}{Re}
\DeclareMathOperator{\im}{Im}
\DeclareMathOperator{\disc}{Disc}
\DeclareMathOperator{\cond}{Cond}
\DeclareMathOperator{\val}{val}
\DeclareMathOperator{\Frob}{Frob}

\newtheorem{lem}{Lemma}[section]
\newtheorem{prop}[lem]{Proposition}
\newtheorem{thm}[lem]{Theorem}

\newtheorem{cor}[lem]{Corollary}

\theoremstyle{definition}

\newtheorem{rem}[lem]{Remark}

\title{The distribution of $\F_q$-points on cyclic $\ell$-covers of genus $g$}
\author{Alina Bucur, Chantal David, Brooke Feigon, Nathan Kaplan, Matilde Lal\'{i}n, Ekin Ozman, Melanie Matchett Wood}
\address{Alina Bucur: Department of Mathematics, University of California at San Diego,
9500 Gilman Drive $\#$0112, La Jolla, CA 92093, USA} \email{alina@math.ucsd.edu}
\address{Chantal David: Department of Mathematics and Statistics,
Concordia University,
1455 de Maisonneuve West,
Montreal, QC H3G 1M8, Canada} \email{cdavid@mathstat.concordia.ca}
 \address{Brooke Feigon: Department of Mathematics,
The City College of New York and CUNY Graduate Center,
NAC 8/133,
New York, NY 10031, USA} \email{bfeigon@ccny.cuny.edu }
\address{Nathan Kaplan: Department of Mathematics, Yale University, 10 Hillhouse Ave,
New Haven, CT 06511, USA}\email{nathan.kaplan@yale.edu}

 \address{Matilde Lal\'in: D\'epartement de math\'ematiques et de statistique,
                                    Universit\'e de Montr\'eal.
                                    CP 6128, succ. Centre-ville.
                                     Montreal, QC H3C 3J7, Canada} \email{mlalin@dms.umontreal.ca}

\address{Ekin Ozman: Bogazici University,
Faculty of Arts and Science,
Department of Mathematics,
34342, Bebek-Istanbul, Turkey } \email{ekin.ozman@boun.edu.tr}

\address{Melanie Matchett Wood: Department of Mathematics, University of Wisconsin-Madison, 480 Lincoln Drive, Madison, WI 53706, USA and
American Institute of Mathematics, 360 Portage Ave,
Palo Alto, CA 94306-2244 USA } \email{mmwood@math.wisc.edu}

\date\today

\begin{document}
\begin{abstract}
We study  fluctuations in the number of points of $\ell$-cyclic covers of the projective line over the finite field $\F_q$ when $q \equiv 1 \mod \ell$ is fixed and the genus tends to infinity.  The distribution is given as a sum of $q+1$ i.i.d. random variables. This was  settled for hyperelliptic curves by Kurlberg and Rudnick \cite{KuRu2009}, while statistics were obtained for certain components of the moduli space of $\ell$-cyclic covers in  \cite{bdfl}. In this paper, we obtain statistics for the distribution of the number of points as the covers vary over the full moduli space of $\ell$-cyclic covers of genus $g$. This is achieved by relating $\ell$-covers to cyclic function field extensions, and  counting such extensions with prescribed ramification and splitting conditions at a finite number of primes.
\\
\\
{\bf Keywords:} curves over finite fields, distribution of number of points, function field extensions, local behavior

\end{abstract}

\maketitle
\tableofcontents

\section{Introduction and results} \label{section-intro}

Let $q$ be a prime power, and let $\F_q$ be the finite field with $q$ elements.
The goal of this paper is to establish statistics for the distribution of the
number of $\F_q$-points of
$\ell$-cyclic covers $C$ of $\PP^1$ defined over $\F_q,$ as $C$ varies over the
moduli space $\mathcal H_{g, \ell}$ of such covers of genus $g$ for large $g$
(and fixed $q$). We always suppose that  $\ell$ is a prime number such that $q
\equiv 1 \pmod \ell.$ For $\ell=2$ (the case of hyperelliptic curves), this was
addressed by Kurlberg and Rudnick \cite{KuRu2009} who showed that the
probability that $\# C(\F_q) = m$ for some integer $m$ is the probability that
the sum of $q+1$ independent and identically distributed (i.i.d.) random
variables is equal to $m$. This was generalized to cyclic $\ell$-covers of
degree $d$ by the first, second, third and fifth named authors in
\cite{bdfl}  who obtained statistics for each irreducible component $\mathcal H^{(d_{1}, \dots, d_{\ell-1})}$ of the
moduli space \begin{eqnarray} \label{modulicomponents}
\mathcal H_{g,\ell}= \bigcup_{\substack{d_{1}+2d_{2} + \dots + (\ell-1)
d_{\ell-1} \equiv 0 \pmod
\ell,\\ 2g= (\ell-1) (d_1+\cdots+d_{\ell-1}-2)}} \mathcal H^{(d_{1}, \dots, d_{\ell-1})},
\end{eqnarray}
as $\min\{d_1, d_2, \dots, d_\ell\}$ tends to infinity. These components will be defined in \ref{subsec:affine}.  Similarly to the hyperelliptic case, the probability that $\#
C(\F_q) = m$ for some integer $m$, as $C$ varies over $\mathcal H^{(d_{1},
\dots, d_{\ell-1})}$ and $\min \{d_1, \dots, d_{\ell-1}\} \rightarrow \infty$, is the
probability that the sum of $q+1$ i.i.d. random variables is equal to $m$. The
i.i.d. random variables $X_1, \dots, X_{q+1}$ are given by (for any prime $\ell
\geq 2$)
\begin{eqnarray} \label{RV}
X_i = \begin{cases} 0 & \textrm{ with probability } \displaystyle \frac{(\ell
-1)q}{\ell(q +\ell -1)},\\
&\\
1 & \textrm{ with probability } \displaystyle \frac{\ell -1}{q+ \ell -1 }, \\
&\\
\ell & \textrm{ with probability }  \displaystyle  \frac{q}{\ell(q +\ell -1)}.
\\
\end{cases}\end{eqnarray}
As the statistics hold for $\min\{d_1, \dots, d_{\ell-1}\} \rightarrow \infty$, this result
does not give statistics for the distribution of the number of $\F_q$-points on covers as we vary over all of
$\mathcal H_{g, \ell}$, since $g \rightarrow \infty$ does not mean that $\min \{d_1,
\dots, d_{\ell-1}\} \rightarrow \infty$ on all components $\mathcal H^{(d_{1},
\dots, d_{\ell-1})}$ for a given genus in \eqref{modulicomponents}. Other
statistics for cyclic $\ell$-covers
were also obtained by counting the covers in a different way (which does not
preserve the genus)
by Xiong \cite{Xi2010} and Cheong, Wood and Zaman \cite{ChWoZa2012}, and the distribution of the number of (affine)
$\F_q$-points on those
covers  was also given by a sum of i.i.d. random variables but with different probabilities than
the random variables of \eqref{RV}.

We show in this paper that the statistics for the distribution of the number of
$\F_q$-points for covers in $\mathcal H_{g, \ell}$ are also given by the random
variables \eqref{RV}. The strategy is completely different from the work in
\cite{bdfl}. There, the counting is done directly by considering affine models for the
covers in each separate component $\mathcal{H}^{(d_1,\dots,d_{\ell-1})}$ of the moduli space.
Here, we study the equivalent question of counting the number of
extensions of the function field $K = \F_q(X)$ with Galois group $\Z/\ell \Z$,
conductor of degree $\od$, and prescribed splitting/ramification conditions at a
finite set of fixed primes of $\F_q(X)$. As a result, we directly obtain the total count
in $\mathcal H_{g, \ell}$. We explain in Section
\ref{section-equivalence} why these two questions are equivalent, and give
general formulas for the number of points on covers in terms of the distribution
of the function field extensions that they define.

In order to count the cyclic function field extensions associated to our
statistics for point counting on covers, we use a classical approach described
by Wright in \cite{Wr1989} (and first due to Cohn \cite{Co54}
 for the case of cubic extensions of the rationals), which is to study the
generating series
\begin{equation} \label{first-GS}
 \sum_{\Gal(L/K) \cong G} \mathfrak{D}(L/K)^{-s},
\end{equation}
where $\mathfrak{D}(L/K)$ is the absolute norm of the discriminant $\disc (L/K)$.
The approach uses class field theory to give an explicit expression for the Dirichlet series
\eqref{first-GS}.
 This is
done in generality by Wright in \cite{Wr1989} for any
global field $K$ and any abelian group $G$. The count is then obtained by an
application of the Tauberian theorem, and the main term is given
by the rightmost pole of the Dirichlet series. The order of this pole varies
according to the group $G$ and the ground field $K$, (more precisely with the
number of roots of unity in $K$). This is described in \cite[Theorem
1.1]{Wr1989}.

In this paper, we apply those techniques to  the case
$K=\F_q(X)$ and
$G = \Z / \ell \Z$, and we further restrict to counting extensions with prescribed splitting conditions at the $\F_q$-rational places of $K$.  To find our desired statistics for point counts of curves, we need to
obtain explicit constants in our asymptotics, and in particular to understand how those constants
change as we change the splitting conditions.  For this, we use the last author's further development of Wright's method in \cite{woodabelian}, which determines probabilities of various splitting types in abelian extensions of number fields.  We are also interested in the secondary terms
and the power saving that can be obtained after taking them into consideration. Our results can then be used to get the
distribution of the number of points on covers as we vary over all of $\mathcal H_{g, \ell}$, but also have other
applications for statistics on the moduli spaces of curves over finite fields, such as the power of traces and
the one-level density. We give more details about these applications in Section \ref{outline}. We also compute the values of the constants for the leading term of the
asymptotic formulas, so the counts obtained with those techniques can be compared with the counts of \cite{bdfl} (see Section \ref{subsec:affine}).

We now state the main results of our paper. We first define some notation. Let   $\mathcal V_K$ be the set of places of $K$.
Let $N(\Z/\ell \Z, \od)$ be the number of extensions of $K=\F_q(X)$ with Galois group
$\Z/\ell \Z$ such that the degree of the conductor is equal to $\od$.
Let $\mathcal{V}_R$, $\mathcal{V}_S$, $\mathcal{V}_I$ denote three finite and disjoint sets of
places of $\F_q(X)$, and let $N(\Z/\ell
\Z, \od; \mathcal{V}_R, \mathcal{V}_S, \mathcal{V}_I)$ be the number of extensions of $\F_q(X)$ with
Galois group $\Z/\ell \Z$, which are ramified at the places of $\mathcal{V}_R$, split at the places of $\mathcal{V}_S$, and inert at the places of $\mathcal{V}_I$, and such that the degree of the conductor is equal to $\od$.

\begin{thm}
 \label{density} \label{thm:main}  Let $\ell \geq 2$ be a prime. Let $\mathcal{V}_R$, $\mathcal{V}_S$, $\mathcal{V}_I$
and $ N(\Z/\ell
\Z, \od; \mathcal{V}_R, \mathcal{V}_S, \mathcal{V}_I)$
 be as defined above, and let ${\mathcal V} = \mathcal{V}_R \cup \mathcal{V}_S \cup \mathcal{V}_I$.
Then, \begin{eqnarray*}
 N(\Z/ \ell \Z, \od) &=& C_\ell \; q^\od P(\od) + O \left( q^{\left( \frac{1}{2} + \varepsilon \right) \od}
\right), \\
 N(\Z/\ell
\Z, \od; \mathcal{V}_R, \mathcal{V}_S, \mathcal{V}_I) &=& C_\ell \left(\prod_{v \in
\mathcal V} c_v\right) \; q^\od P_{\mathcal{V}_R,\mathcal{V}_S, \mathcal{V}_I }(\od)  +  O \left( q^{\left( \frac{1}{2} + \varepsilon \right) \od}
\right),
 \end{eqnarray*}
where $P(X), P_{\mathcal{V}_R,\mathcal{V}_S, \mathcal{V}_I }(X)  \in \R[X]$ are monic polynomials of degree
$\ell-2$. Furthermore,  $C_\ell$ is the non-zero constant given by
\begin{eqnarray} \label{constantC0}
 C_\ell =   \frac{(1-q^{-2})^{\ell-1}}{(\ell-2)!}
\prod_{j=1}^{\ell-2}  \prod_{v\in \mathcal{V}_K} \left(1-\frac{jq^{-2\deg v}}{(1+q^{-\deg
v})(1+jq^{-\deg v})}\right) , \end{eqnarray}
and  for each place $v\in \mathcal V$, we have
$$c_{v} = \begin{cases} \displaystyle  \frac{(\ell-1)q^{-\deg
v}}{1+(\ell-1)q^{-\deg v}} & \textrm{if $v \in {\mathcal V}_R$},\\ \\
\displaystyle \frac{1}{\ell(1+(\ell-1)q^{-\deg{v}})} & \textrm{if $v \in {\mathcal V}_S$, }\\ \\
\displaystyle \frac{\ell-1}{\ell(1+(\ell-1)q^{-\deg{v}})} & \textrm{if $v \in {\mathcal V}_I$. }\end{cases}$$
Furthermore, for $\ell=2$ we obtain the exact count
\begin{eqnarray} \nonumber
N(\Z/2\Z,\od) &=&  \begin{cases} 2(q^\od - q^{\od-2}) & \od >2, \od \textrm{ even,}\\
2q^2 & \od =2, \\
0 & \od \textrm{ odd.} \end{cases}\end{eqnarray}
\[N(\Z/2\Z, \od, v_0, \mbox{ramified})=\frac{(1-q^{-2})}{1+q^{-\deg v_0}} q^{\od-\deg v_0}+O_q(1).\]
\end{thm}

We prove Theorem \ref{thm:main} by using class field theory to show that counting $\Z/\ell\Z$ extensions of $\F_q(X)$ is equivalent to counting continuous
homomorphisms  of the id\`{e}le class group of $\F_q(X)$ to  $\Z/\ell \Z$. This
is the method implemented by \cite{Wr1989} for general abelian
extensions over function fields
and number fields, and
also in some recent work of the last author \cite{woodabelian} that finds probabilities of various splitting types in abelian extensions of number fields.
The idea of obtaining statistics for the families of curves over finite fields by considering the family of function field extensions attached to those curves was also used by Wood in \cite{Wo2012} and by Thorne and Xiong in \cite{ThXI} for the family of trigonal curves (corresponding to non-Galois cubic extensions of $\F_q(X)$).

We record below a special case of this result which will be needed in the
applications described in Section \ref{outline}.
The following corollary has a necessary ingredient for proving such results, namely, the
explicit dependence of each of the coefficients of the polynomial $P_{\mathcal{V}_R,\mathcal{V}_S, \mathcal{V}_I }(X) $ with respect to the splitting/ramification conditions to ensure enough
cancellation in the relative densities for the split and inert primes. More
corollaries of this type can be extracted from the proof
of Theorem   \ref{thm:main} if needed for other applications.

\begin{cor}
   \label{thm:specialcase} Let $v \in \mathcal{V}_K$ be a place, let
$\epsilon \in  \{$ramified, split, inert$\}$,
   and let
                                     $N(\Z/ \ell \Z, \od, v, \epsilon)$ be
the number of extensions of $\F_q(X)$ with Galois group $\Z/\ell \Z$  such that the degree of the conductor is equal to $\od$ and with prescribed behavior $\epsilon$ at the
   place $v.$
   Then,
   \begin{eqnarray*}
   N(\Z/ \ell \Z, \od, v, \textrm{ramified}) &=& \frac{(\ell-1)q^{-\deg v}}{1+(\ell-1)q^{-\deg v}}C_\ell q^\od P_R(\od) + O
\left( q^{\left( \frac{1}{2} + \varepsilon \right) \od} \right), \\
    N(\Z/\ell \Z , \od, v, \textrm{split}) &=& \frac{1}{\ell(1+(\ell-1)q^{-\deg v})}C_\ell q^\od P_S(\od)
    + O \left( q^{\left( \frac{1}{2} + \varepsilon \right) \od} \right), \\
    N(\Z/ \ell \Z, \od, v, \textrm{inert}) &=& \frac{1}{\ell(1+(\ell-1)q^{-\deg v})}C_\ell q^\od P_I(\od) + O \left( q^{\left( \frac{1}{2} + \varepsilon \right) \od} \right) ,
   \end{eqnarray*}
   where $C_\ell$ is the non-zero constant defined by \eqref{constantC0},
  $P_R(X)$ and $P_S(X) \in \R[X]$ are monic polynomials of degree $\ell-2$ and $P_I(X)=(\ell-1)P_S(X)$.

 When $\ell=2$, we obtain a better error term for the ramified case
\[N(\Z/2\Z, \od, v, \mbox{ramified})=\frac{(1-q^{-2})q^{-\deg v}}{1+q^{-\deg v}} q^{\od}+O_q(1).\]
  \end{cor}

Finally, we state our main result for the distribution of points on $\ell$-cyclic
covers of $\mathbb P^1$ of fixed genus that can be obtained by a simple application of Theorem
\ref{thm:main}.  This distribution is given in terms of the same random variables obtained in \cite{bdfl} for each irreducible component of the moduli space  $\mathcal H_{g, \ell}$.

\begin{thm} \label{thm-HG}
Let $\mathcal H_{g, \ell}$ be the moduli space of $\Z/\ell \Z$ Galois covers of
$\mathbb P^1$ of genus $g$. Then, as $g \rightarrow \infty$, \[\frac{|\{C \in
\mathcal H_{g, \ell}(\F_q) :  \# C(\F_q) = m \}|'}{|\mathcal H_{g, \ell}(\F_q)|'} =
 \Prob \left(X_1+ \dots X_{q+1} = m \right) + O_\ell \left( \frac{1}{g} \right), \]
where the $X_i$'s are independent identically distributed random variables such that
\[X_i = \begin{cases} 0 & \textrm{ with probability } \displaystyle \frac{(\ell
-1)q}{\ell(q +\ell -1)},\\
&\\
1 & \textrm{ with probability } \displaystyle \frac{\ell -1}{q+ \ell -1 }, \\
&\\
\ell & \textrm{ with probability }  \displaystyle  \frac{q}{\ell(q +\ell -1)}.
\\
\end{cases}\]
In the formula, as usual, the $'$ notation means that the covers $C$ on the moduli space are counted with the usual wights $1/|\mathrm{Aut}(C)|$.
\end{thm}

\subsection{Relation to previous work and outline of the paper} \label{outline}

As we mentioned above, using class field theory to count Abelian extensions of global fields
was first used in the elegant note of Cohn \cite{Co54} for the particular case $K=\Q$ and $G = \Z / 3 \Z$, and was vastly generalized
by Wright  in his influential paper on the subject \cite{Wr1989}. The main idea is
to write the generating series \eqref{first-GS}
as a finite linear combination of Euler products whose factors are relatively simple.
In  \cite{Wr1989}, Wright gets an asymptotic for all Abelian extensions of a global
field. In the present paper, we are interested in a special case of his work, namely $K=\F_q(X)$ and
$G = \Z/\ell \Z$, but we need results which are completely explicit because of the applications to
statistics of curves over finite fields, which is our main goal.
We then need the values of the constants
$c(k, G)$ in \cite[Theorem I.3]{Wr1989}, which are not determined by Wright. He manages by an
ingenious argument to show that they are non-zero, and that the density exists. In this paper, we compute
these constants explicitly and show that they fit the count of \cite{bdfl}, ignoring the error terms (see Section \ref{section-equivalence}). For the case of number field extensions, the explicit computation of the constants $c(k, G)$ from \cite[Theorem I.3]{Wr1989} was addressed by Cohen, Diaz y Diaz and Olivier in \cite{CDO}, again for the case of cyclic extensions of prime degree. Their techniques are completely different from the class field theory approach of \cite{Co54, Wr1989}, as they use Kummer theory. The authors of \cite{CDO} do not compute the relative densities for the splitting conditions with their approach, and to our knowledge, this is not done in the current literature with the Kummer theory approach.

The relative densities for general Abelian extensions of number fields were computed explicitly by the last author of the present paper in \cite{woodabelian} using the class field theory approach, with an emphasis on characterizing the extensions of $\Q$ where the independence between the various primes in the relative densities is false. In some ways, the present paper is a function field analogue of \cite{woodabelian},
but of course the application  for counting points on curves is different. There are also many differences between function fields and number fields,  because of
the special role of the place at infinity, and the fact that all residue fields have the same characteristic. There are also different analytic issues between number fields and function fields.

Among other possible applications of counting function fields extensions and curves over finite fields, one can think of statistics for the distribution of points over finite fields $\F_{q^n}$ as $n$ varies (but the family of covers is still defined over $\F_q$),
and the one-level density for the family, as studied by Rudnick in \cite{Ru-hyper} for hyperelliptic curves. Those applications are also suitable for an approach using the relative densities of the function field extensions corresponding to
the family of curves.
For this particular application, one needs all secondary terms which can be obtained by computing
the residues of all the poles in the line $\re(s)=1/(\ell-1)$ of the generating series \eqref{first-GS}, i.e. the polynomials $P(\od)$ appearing in Theorem \ref{thm:main} and Corollary \ref{thm:specialcase},
and not only the main term given by the highest order pole. This would provide enough cancelation between the different relative densities appearing
in the explicit formulas relating the point counting to the zeroes of the zeta functions of the curves.
The quality of the results obtained for statistics for the number of points over $\F_{q^n}$
(namely how large  $n$ is with respect to the
genus of the family) and the one-level density (namely the support of the Fourier transform) is influenced by the error term, which is obtained from the Tauberian theorem
after considering the poles as discussed. One important feature of the error term is its dependence on the
degree of the primes with splitting/ramification conditions. This dependence raises delicate and nontrivial issues, as there are cyclic $\ell$-covers where
the zeroes of the zeta functions are related and their contribution to the error term is large, but they should not influence the average over the family as they are exceptional.
These questions are not addressed in this paper, but are being
considered in work in progress.

Finally, we say a few words about our restriction to $q \equiv 1 \mod \ell$.
We are interested in counting points on the curves $Y^\ell = F(X)$ defined over $\F_q$. If $q \not\equiv 1 \mod \ell$, the point counting on the curves is trivial as every element in $\F_q$ is an $\ell$th power (in a unique way). Also, if $q \not\equiv  1 \mod \ell$, then the extension corresponding to the curve $Y^\ell = F(X)$ is not a cyclic extension,
and  the cyclic extensions of  $\F_q(X)$ of degree $\ell$ do not come from those curves in the case where $q \not\equiv 1 \mod \ell$.

We now outline the organization of the paper.
In Section \ref{section:setup}, we establish the notation and use class field theory to translate the counting of extensions to the counting of maps of the id\`{e}le
class group. We also prove a general form of the Tauberian theorem over function
fields that we  need to analyze the
Dirichlet series  for cyclic extensions of $\F_q(X)$ that is a slight generalization of a result in \cite{rosen}. In Section \ref{section-Dirichlet}, we define
Dirichlet characters over $\F_q(X)$, and we prove
analytic properties of some Dirichlet series that appear in future
sections.
In Section \ref{section:caseell}, we prove our main result, Theorem \ref{thm:main}. In Subsection
 \ref{section:case2}, we look at the particular case of $\ell=2$ where we can get the exact result for the
the total number of quadratic extensions with fixed conductor, and the case with one prescribed ramified place
with a better error term without using the Tauberian theorem. Finally, we
explain in Section \ref{section-equivalence} how to obtain statistics for the
point counting over the moduli space of cyclic $\ell$-covers, and we compare our
results with those of \cite{bdfl}.

\section{Background and setup}\label{section:setup}

In this section, we set up notation and recall basic facts from Galois theory
and class field theory that allow us to rephrase our problem in terms of
counting continuous homomorphisms from the id\`ele class group of a function
field to a cyclic group of prime order.

 Fix a prime  $\ell$. Throughout the paper $\F_q$ denotes a finite field with $q
\equiv 1 \pmod \ell$ elements and $K=\F_q(X)$ is the rational function field
over $\F_q$.

\subsection{Notation}  We will denote by $G_K$  the absolute Galois group of $K,$ that is the Galois
group $\Gal(K^{\textrm{sep}}/K)$ of the separable closure of $K.$ Let $\CCD_K^+$
be the set of effective divisors of $K.$ For each place $v$ of $K$ we will use
the standard notations $K_v$ for the completion at $v,$ $\mathcal O_v$ for the
local ring, $\kappa_v$ for the residue field, and $\pi_v$ for a uniformizer at
$v$ which we choose to be monic. Recall that the degree of a place $v$ is given by $\deg v = [\kappa_v:\F_q]$
and its norm is $Nv = q^{\deg v}$, the number of elements in the residue field
$\kappa_v.$ Of course, for a place $v_f$ associated to an irreducible polynomial
$f \in \F_q[X]$,
 we have that $\deg v = \deg f$. For the place at infinity associated with
the uniformizer $\pi_\infty = 1/X$, we have that $\deg v_\infty = 1$.

\subsection{From covers to field extensions}\label{sec:coverext}
A $\Z/\ell\Z$ cover is a pair $(C, \pi)$ where $C\stackrel{\pi}{\to}\PP^1$ is an $\ell$-degree cover map defined over $K.$
Each $\Z/\ell\Z$ cover $(C,\pi)$ together with an isomorphism $\Z/\ell\Z \to \Aut(C/\PP^1)$ corresponds to a Galois extension $L$ of $K = \F_q(X)$ together with a distinguished isomorphism $\Gal(L/K) \stackrel{\tau}{\to} \Z/\ell \Z$.
We refer to such extensions as $\ell$-cyclic extensions.  The genus of the curve $C$ is related to the discriminant  $\disc (L/K)$ via the Riemann-Hurwitz formula (see for instance \cite[Theorem 7.16]{rosen}),
\[ 2 g_C -2 = \ell  (2 g_{\PP^1} -2) + \deg \disc (L/K).\]

Since $q \equiv 1 \pmod \ell,$ there is no wild ramification and each place $v$ of $K$ either ramifies completely, splits completely or is inert. Thus \begin{equation}\label{eq:disccond}\disc (L/K) = \sum_{v \textrm{ ramified in } L} (\ell -1) v\end{equation}
and  \[2g_C =   (\ell -1) \left[-2 + \sum_{v \textrm{ ramified in }L} \deg v\right],\] where the sum is taken over the places $v$ of $K$ that ramify in $L.$

\subsection{From field extensions to maps}\label{subsection:maps} Our translation from counting extensions to counting maps has two steps.
 First, by Galois theory, $\ell$-cyclic extensions $L/K$ with a distinguished isomorphism $\Gal(L/K) \stackrel{\tau}{\to} \Z/\ell \Z$ are in one-to-one correspondence with the
surjective continuous homomorphisms $G_K \to \Z/\ell \Z$ from the absolute Galois group
of $K$ to $\Z/\ell\Z.$  By class field theory, the maps $G_K \to \Z/\ell \Z$ are
in one-to-one correspondence with the maps $\mathbf J_K/K^\times \to \Z/\ell \Z$
from the  id\`ele class group of  $K$ to $\Z/\ell \Z.$

Since $K$ contains all $\ell$th roots of unity and $(\ell,q)=1$ each $\ell$-cyclic Galois cover is of the form $K(\sqrt[\ell]{\beta}), \dots, K(\sqrt[\ell]{\beta^{\ell -1}})$ by Kummer theory. These correspond to  $\ell -1$ unramified surjective continuous homomorphisms $\mathbf J_K/K^\times \to \Z/\ell \Z$, one for each generator of $\Z/\ell \Z$. There is also the trivial map, which is also unramified everywhere. In terms of extensions, this corresponds to the $K$-algebra $K^\ell$. In terms of covers of $\PP^1$, this corresponds to the split cover that
consists of $\ell$ disjoint copies of $\PP^1.$

Thus an $\ell$-cyclic extension $L/K$ of given discriminant corresponds to a nontrivial
continuous  homomorphism $\varphi: {\mathbf J}_K/K^\times \to \Z/\ell \Z.$

 Let $\phi$ be a map 
\begin{eqnarray}\label{countmaps}
 \phi: \pi_\infty ^\Z \times \prod_v \co_v^\times \to  \Z/\ell\Z
\end{eqnarray}
which is trivial on the embedding of $\F_q^\times$ in $\prod_v \co_v^\times$.  Here $\pi_\infty ^\Z$ is the free abelian group generated by $\pi_\infty$ and the product is taken over all the places $v$ including the place at infinity\footnote{Unless otherwise specified, we will continue using the convention  that the sums and products over $v$ denote all places including the place at infinity.}.

 \begin{rem}We remark that $\phi$ and $\varphi$ are two different maps.  \end{rem}
 
 The maps $\phi$ and $\varphi$ are closely related via the following proposition whose proof can be found in \cite[Page 90, Section 7]{hayes}.

\begin{prop}\label{unram}
 Let
\[
 \phi=\psi_\infty \otimes_v \phi_v: \pi_\infty ^\Z \times \prod_v \co_v^\times \to  \Z/\ell\Z.
\]
Then\begin{enumerate}
\item   If $\phi$ is trivial on   the embedding of $\F_q^\times$ in $\prod_v \co_v^\times$  it has a
unique extension to a map
\[
\varphi:{\mathbf J}_K/K^\times \to \Z/\ell\Z.
\]
  \item
A place $v$ of $K$ ramifies in an $\ell$-cyclic extension $L$  corresponding to $\phi$
 if and only if the map $\phi_{v}$ is nontrivial on $\co_v^\times$.
\end{enumerate}
\end{prop}

Thus the conductor of the map $\phi$ is
$$\cond(\phi) = \displaystyle\sum_{v \textrm{ ramified in } L}  v,$$
which is also the conductor of the extension $L/K$. As there is no wild ramification, the
discriminant-conductor formula (see for instance \cite[Section 12.6]{Sal}) yields
 \begin{eqnarray} \label{disccond1}
\disc(L/K) = (\ell-1) \cond(L/K) = (\ell -1) \cond
(\phi).
\end{eqnarray}

In section \ref{section:caseell} we prove our main results by working with $\phi$. For the remainder of this section we explicate the relationship between $\phi$ and  the corresponding $L/K$. First we   address the global compatibility condition needed for $\phi$ to  extend to a function $\varphi$ defined over ${\bf J}_K/K^\times$ and hence correspond to an extension $L$. Namely, that $\phi$ must be trivial on the embedding of $\F_q^\times$ in $\prod_v \co_v^\times$.  Fix $\mu
\in \F_q,$ a generator of the multiplicative group $\F_q^\times.$  Clearly $\phi$ is trivial on $\F_q^\times$ if and only if $\phi(1, \mu, \mu, \dots)=0$ where the first component in the infinite vector corresponds to the identity element in the free abelian group $ \pi_\infty ^\Z$.

For each place
$v$ of $K,$ we note that the map $\phi_v : \mathcal O_v^\times \to \Z/\ell\Z$ factors through
$\mathcal O_v^\times / (1+\pi_v\mathcal O_v)\cong \left(\mathcal O_v / (\pi_v) \right)^\times.$ Recall that $\deg v=[\mathcal O_v/( \pi_v ): \mathbb F_q]$ and thus
\[\mathcal O_v/( \pi_v )\cong \F_{q^{\deg v}}.\]
For each $v$, fix a choice of $g_v \in \mathcal O_v$  whose image generates $\mathcal O_v^\times / (1+\pi_v\mathcal O_v)\cong (\F_{q^{\deg v}})^\times$ and such that
\[\mu =g_v^{\frac{q^{\deg v}-1}{q-1}}.\]
Then\begin{align*}
\phi(1, \mu, \mu, \dots)&=\phi((1,\mu, 1, 1, \dots)(1,1, \mu, 1, \dots)\cdots)\\
&=\phi(1,\mu, 1, 1, \dots)+\phi(1, 1, \mu, 1, \dots)+\dots\\
&=\sum_v \phi_v(\mu)
=\sum_v\phi_v\left(g_v^{\frac{q^{\deg v}-1}{q-1}}\right)=\sum_v \left(  \frac{q^{\deg v}-1}{q-1} \right)\phi_v(g_v).
\end{align*}
We note that $\frac{q^{\deg v}-1}{q-1}=q^{\deg v-1}+q^{\deg v-2}+\dots +q+1
\equiv \deg v \pmod \ell$ since $q \equiv 1 \pmod \ell$.
We have now proved the following proposition.

\begin{prop}\label{globalccresult}
For each $v$ let $g_v \in \mathcal{O}_v$ as defined above. A map $\phi: \pi_\infty ^\Z \times \prod_v \co_v^\times \to  \Z/\ell\Z$ is trivial on the embedding of  $\F_q^\times$ in $\prod_v \co_v^\times$  if and only if 

\begin{equation}\label{eq:modell}
\sum_{v \in \textrm{Cond}(\phi)}  \phi_v(g_v) \, \deg v \equiv 0 \pmod \ell.
\end{equation}

\end{prop}

Thus, in order to count the extensions $L/K$ with prescribed
splitting/ramification conditions at places $v$ of $K=\F_q(X)$, it is necessary and sufficient to count
the maps $\phi$ as in \eqref{countmaps} satisfying the global compatibility condition \eqref{eq:modell} with corresponding conditions at places $v$ of $K$, which we describe below.
By Proposition \ref{unram}  a place $v$ is ramified if and only if  $\phi_v$ is nontrivial on $\co_v^\times$. Now we deal with the remaining two cases, inert and completely split.

\begin{prop}\label{localmaps}
Let $$\phi=\psi_\infty \otimes_v \phi_v: \pi_\infty ^\Z \times \prod_v \co_v^\times \to  \Z/\ell\Z$$ and trivial on the embedding of  $\F_q^\times$ in $\prod_v \co_v^\times$.  Let $$\varphi:{\mathbf J}_K/K^\times \to \Z/\ell\Z$$ be the unique extension of $\phi$ to the id\`ele class group of $K$ and let $L$ be the $\ell$-cyclic extension of $K$  corresponding to $\varphi$. Let $v_0$ be a place of $K$ different from $v_\infty$. Then we have the following: 

\begin{enumerate}
\item
$v=v_0$ or $v_\infty$ ramifies in  $L$  
 if and only if the map $\phi_{v}$ is nontrivial on $\co_v^\times$,
\item \label{test}  $v_0$ splits completely in $L$ if and only if $\phi_{v_0}(\mathcal O_{v_0}^\times)=0$  and
\begin{align}\label{eqn:magical}
\psi_\infty(\pi_\infty^{-\deg v_0})+\sum_{v\neq v_0, v_\infty}\phi_v(\pi_{v_0})=0,
\end{align}
\item  $v_0$ is inert in $L$ if and only if $\phi_{v_0}(\mathcal O_{v_0}^\times)=0$  and
\[
\psi_\infty(\pi_\infty^{-\deg v_0})+\sum_{v\neq v_0, v_\infty}\phi_v(\pi_{v_0})\neq 0,
\]
\item $v_\infty$ splits
completely in $L$ if and only if $\phi_{v_\infty}(\co_{v_\infty}^\times) = 0$ and $\psi_\infty(\pi_\infty) = 0,$
\item $v_{\infty} $ is inert in $L$ if and only if $\phi_{v_\infty}(\co_{v_\infty}^\times) = 0$ and  $\psi_\infty(\pi_\infty)  \neq 0.$
\end{enumerate}
\end{prop}

 \begin{proof}
 Let $\varphi_v$ be the composition of $\varphi$ with the canonical map $K_v^\times\rightarrow {\mathbf J}_K \rightarrow {\mathbf J}_K/K^\times$.
In the case where $v$ is unramified, the map $\varphi_v$ is trivial on $\co_v^\times$ and
therefore its image is dictated by $\varphi_v(\pi_v^\Z)\in \Z/\ell\Z.$ Thus the image  is a subgroup of a simple abelian group and we have only
two possibilities: either $\varphi_v$ is surjective or $\varphi_v$ is trivial.
Since $\Frob_{v}$ corresponds to the vector with $\pi_v$ in the $v$ place and $1$ elsewhere under the correspondence from class field theory, $v$ splits  if and only if $\varphi_v(\pi_{v})=0$.

Now, let $v_0 \neq v_\infty$ be unramified. For the purpose of this particular discussion we denote elements in the id\`eles by vectors with the infinite component first 
and the $v_0$ component second. Under this notation   $v_0$ splits  if and only if $\varphi(1,\pi_{v_0}, 1,1, \dots)=0$. 
Since $\varphi$ is trivial on $K^\times$, we know that
\begin{align*}
0=&\varphi(\pi_{v_0},\pi_{v_0}, \dots ) =\varphi(\pi_{v_0}, 1, \dots )+ \varphi(1, \pi_{v_0}, 1, \dots )+ \varphi(1,1,\pi_{v_0},1,\dots )+\dots
\\
=&\varphi(\pi_\infty^{-\deg v_0}, 1, \dots ) +\varphi(\pi_{v_0}\pi_\infty^{\deg v_0}, 1,\dots ) + \varphi(1, \pi_{v_0}, 1, \dots )+
\\
&+\varphi(1,1,\pi_{v_0},1,\dots )+\varphi(1,1,1,\pi_{v_0},1, \dots)+\dots.
\end{align*}

Since we chose  $\pi_{v_0}$ to be monic    and $\val_{\infty}(\pi_{v_0}\pi_{\infty}^{\deg v_0})=0$,  we have that $\varphi(\pi_{v_0}\pi_\infty^{\deg v_0}, 1,\dots)=0$. Denoting by $\varphi_{v_0}(\pi_{v_0})$ the term $\varphi(1,\pi_{v_0}, 1, \dots, 1)$ where we recall the convention that the second place 
corresponds to $v_0$, we obtain,
\[
\psi_\infty(\pi_\infty^{-\deg v_0})+\varphi_{v_0}(\pi_{v_0})+\sum_{v\neq v_0, v_\infty}\phi_v(\pi_{v_0})=0.
\]
Since $v_0$ splits if and only if $\varphi_{v_0}(\pi_{v_0})=0$, we see that:

\begin{itemize}
\item $v_0$ splits if and only if $\phi_{v_0}(\mathcal O_{v_0}^\times)=0$  and
\begin{align}
\psi_\infty(\pi_\infty^{-\deg v_0})+\sum_{v\neq v_0, v_\infty}\phi_v(\pi_{v_0})=0.
\end{align}
\item  $v_0$ is inert if and only if $\phi_{v_0}(\mathcal O_{v_0}^\times)=0$  and
\[
\psi_\infty(\pi_\infty^{-\deg v_0})+\sum_{v\neq v_0, v_\infty}\phi_v(\pi_{v_0})\neq 0.
\]
\end{itemize}

If $v = v_\infty$ we can read the splitting behavior from $\phi(\pi_\infty, 1,
1, \dots).$ Namely, we have that $v_\infty \notin \cond(\phi)$ if and only if
$\phi_{v_\infty}(\co_{v_\infty}^\times) = 0.$  Therefore:
\begin{itemize}
\item $v_\infty$ splits
completely in $L$ when $\phi_{v_\infty}(\co_{v_\infty}^\times) = 0$ and $\psi_\infty(\pi_\infty) = 0,$
\item $v_{\infty} $ is inert when $\phi_{v_\infty}(\co_{v_\infty}^\times) = 0$ and  $\psi_\infty(\pi_\infty)  \neq 0.$
\end{itemize}

\end{proof}

\subsection{Generating series and the Tauberian Theorem}
\label{section:tauberian}

As in previous work, our strategy is to make use of the Tauberian theorem to deduce
an asymptotic formula for the number of field extensions $L/K$ with discriminant
of degree $\od$ from the analytic properties of the generating series
$$\sum_{\Gal(L/K) \cong \Z/\ell\Z}
\mathfrak{D}(L/K)^{-s},$$
where $\mathfrak{D}(L/K)$ is the norm of the discriminant $\disc(L/K)$.
As mentioned above, since we are dealing with cyclic extension of prime degree
$\ell$, the conductor-discriminant relation gives
$$\disc(L/K) = (\ell-1) \cond(L/K) \iff \mathfrak{D}(L/K) = N \left(
\cond(L/K) \right)^{\ell-1},$$
and it is more natural to write the generating series as
$$ \sum_{\Gal(L/K) \cong \Z/\ell\Z} \mathfrak{D}(L/K)^{-s}
:=
\sum_{f \in \mathcal{D}_K^+} \frac{a_\ell(f)}{N\!f^{(\ell-1)s}},$$
where $a_\ell(f)$ is the number of cyclic extensions of degree $\ell$ of
$K=\F_q(X)$ with conductor $f$.
We will then extend this analysis to study the extensions $L$ that are counted by $N(\Z/\ell \Z,
\od; {\mathcal{V}}_R, {\mathcal{V}}_S, {\mathcal V}_I)$ as defined in Section \ref{section-intro} by
understanding the generating series
\[
\sum_{{\Gal(L/K) \cong \Z/\ell\Z}\atop{ {\mathcal{V}}_R, {\mathcal{V}}_S, {\mathcal V}_I}}
\mathfrak{D}(L/K)^{-s},
\]
where the sum now runs over the cyclic extensions of degree $\ell$ that satisfy
all of prescribed splitting/ramification conditions at the places of ${\mathcal{V}}_R\cup{\mathcal{V}}_S\cup {\mathcal V}_I$. Again, we will write this Dirichlet series as
\begin{eqnarray*}
\sum_{{\Gal(L/K) \cong \Z/\ell\Z}\atop{ {\mathcal{V}}_R, {\mathcal{V}}_S, {\mathcal V}_I}}
\mathfrak{D}(L/K)^{-s} :=
\sum_{f \in \mathcal{D}_K^+}  \frac{a_\ell(f, \mathcal{V}_R, \mathcal{V}_S,
\mathcal{V}_I)}{N\!f^{(\ell-1)s}},
\end{eqnarray*}
where $a_\ell(f, \mathcal{V}_R, \mathcal{V}_S,
\mathcal{V}_I)$ is the number of cyclic extensions
of degree $\ell$ of $K=\F_q(X)$ with conductor $f$ that satisfy all of the prescribed
splitting/ramification conditions.

We now state and prove
the version of the Tauberian theorem needed to analyze the Dirichlet series
above. More generally, let $k$ be a positive integer,
let $a : \CCD_K^+ \to \C$, and $\CF(s)$ be the Dirichlet series
 $$
 \CF(s) =\sum_{f \in \CCD_K^+} \frac{a(f)}{N\!f^{ks}}.$$

We need a Tauberian theorem that will allow us to evaluate $\sum_{\deg f = \od}
a(f)$ in the situation when the half-plane of absolute convergence is
$\re(s) > 1/k$ for some positive integer $k$,  and the function $\CF(s)$ has a finite number of poles (of arbitrary multiplicities) on the line $\re(s) =
1/k$. This is a slight  generalization of \cite[Theorem 17.1]{rosen}.

Since the function $q^{-ks}$, and therefore $\CF(s)$, are periodic with
period ${2 \pi i}/{( k\log{q})}$, nothing is
lost by confining our attention to the region
\begin{eqnarray}\label{B_k}B_k = \left\{ s \in \C :  - \frac{\pi i}{k\log{q}} \le \im(s) < \frac{\pi
i}{k\log{q}} \right\}. \end{eqnarray}
We will always suppose that $s$ is confined to the region
$B_k$.

\begin{thm}\label{thm:Tauberian} Let $k$ be a positive integer, and let $0 <
\delta < 1/k$. Let $a : \CCD_K^+ \to \C$, and suppose that the Dirichlet series
 $$
\mathcal{F}(s) = \sum_{f \in \CCD_K^+} \frac{a(f)}{N\!f^{ks}}$$
converges
absolutely for $\re(s) > 1/k$, and is holomorphic on $\left\{s \in B_k : \re(s)
\geq \delta \right\}$ except for a finite number of poles on the line $\re(s) = 1/k$.
Let $u=q^{-ks}$ and define $F(u) = \CF(s).$
Then,
$$ \sum_{\deg f= \od} a(f)  = - \sum_{|u|=q^{-1}} \Res_{u}
\frac{F(u)}{u^{\od+1}} + O \left( q^{\delta k \od} M \right),$$
where $$M = \max_{|u| = q^{- k \delta}} |F(u)| = \max_{\re(s)=\delta} |\mathcal{F}(s)|.$$
\end{thm}

\begin{proof}
With the change of variable $u = q^{-ks}$, we have that
$$
F(u) = \sum_{\od=0}^{\infty} \left( \sum_{\deg{f} = \od} a(f) \right) u^{\od},
$$
and by hypothesis,  $F(u)$ is a meromorphic function on the disk $\{u \in \C:
|u|\leq q^{-k \delta}\}$, except for finitely many poles with $|u|=1/q$. Let
$C_\delta=\{u \in \C : |u|=q^{-k\delta}\}$,
oriented counterclockwise. Choose any
$\eta>1$ and let  $C_\eta=\{u \in \C : |u|=q^{-\eta}\}$, oriented clockwise.
Notice that $\frac{F(u)}{u^{\od+1}}$ is a meromorphic function between the two
circles
$C_\eta$ and $C_\delta$ with finitely many poles at $|u|=1/q$.
Thus, by the Cauchy's integral formula,
\[\frac{1}{2\pi i} \oint_{C_\delta+C_\eta} \frac{F(u)}{u^{\od+1}} d u =
\sum_{|u| = q^{-1}} \Res_{u} \frac{F(u)}{u^{\od+1}}.\]
Since $q^{-\eta} < q^{-1}$, using the power series expansion of $F(u)$ around $u=0$,
we have that
\[\frac{1}{2\pi i} \oint_{C_\eta} \frac{F(u)}{u^{\od+1}} d u=-\sum_{\deg f =
\od} a(f).\]
Therefore, we obtain
\[\sum_{\deg f =\od} a(f)= - \sum_{|u| = q^{-1}} \Res_{u} \frac{F(u)}{u^{\od+1}}
 + \frac{1}{2\pi i} \oint_{C_\delta} \frac{F(u)}{u^{\od+1}} d u.\]
Let $M$ be the maximum of $|F(u)|$ over $C_\delta$. Then
\[\left| \frac{1}{2\pi i} \oint_{C_\delta} \frac{F(u)}{u^{\od+1}} d u \right|
\leq M q^{\delta k \od},\]
which proves the result.
\end{proof}

\section{Dirichlet characters and $L$-functions}\label{characters}
\label{section-Dirichlet}

In this section, we define $\ell$th-power residue symbols over $\F_q[X]$. We refer the
reader to \cite{moreno, rosen} for details. We then study the convergence properties of some auxiliary functions built out of the $\ell$th-power residue symbols
that will be used in the proofs of our main results.

Recall that $\ell$ is a prime  such that $q \equiv 1 \pmod \ell$. Thus $\F_q^\times$ contains the $\ell$th roots of unity. In particular, $b_\ell=\mu^\frac{q-1}{\ell}$ is one of these roots where $\mu$ is a fixed generator of $\F_q^\times$.
For a place $v\not = v_\infty$ of $K$, we also let  $v=v(X) \in \F_q[X]$ represent the monic irreducible 
polynomial in $K^\times$ corresponding to $v$. We
define the $\ell$th power residue symbol as follows. Let
\[
\left( \frac {\cdot}{v}\right)_\ell : (\F_q[X]/v(X))^\times \rightarrow \F_q^\times
\]
be defined by
\[
\left( \frac {f}{v}\right)_\ell\equiv f^{\frac{Nv-1}{\ell}} \pmod v.
\]
In other words, the $\ell$th power residue symbol is given by an $\ell$th root of unity.

Recall that the choice of $\mu$ made in Section \ref{subsection:maps} determined for each place $v$ a generator $g_v$  of
 \[
(\mathcal{O}_v/(\pi_v))^\times\cong (\F_q[X]/(v(X)))^\times \cong (\F_{q^{\deg v}})^\times
\]
such that $\mu=g_v^{\frac{q^{\deg v}-1}{q-1}}$. We have
\[g_v^{\frac{q^{\deg v}-1}{\ell}} = \left(g_v^{\frac{q^{\deg v}-1}{q-1}}\right)^{\frac{q-1}{\ell}}=\mu^{\frac{q-1}{\ell}}=b_\ell.\]
By the definition of the $\ell$th power symbol,
\[\left(\frac{g_v}{v}\right)_\ell\equiv b_\ell \pmod v.\]
We let $\sigma$ be an $\ell$-order character from $\F_q^\times\to\mathbb C^\times$.
Then,
\[\chi_{v,\ell} := \sigma\circ  \left( \frac {\cdot}{v}\right)_\ell\]
is a Dirichlet character $\chi: \F_q[X] \rightarrow \C^\times$ of modulus $v$, where we
define $\chi_{v,\ell}(f(x)) = 0$ if $v(x)$ divides $f(x)$.

For the infinite place $v_\infty$, we further define
\begin{equation}\label{chi(vinfty)}
\chi_{v, \ell}(v_\infty)=\begin{cases} 1 & \deg v \equiv 0 \pmod \ell,
\\
0 & \deg v \not \equiv 0 \pmod \ell.
\end{cases}
\end{equation}

For $\chi$ a nontrivial Dirichlet character, we denote by  $L(s, \chi)$ the
Dirichlet  $L$-function
\[L(s, \chi) = \sum_{\substack{F \in \F_q[X] \\F \text{ monic } }} \frac{\chi(F)}{|F|^s}\]
where $F$ varies over the monic polynomials of $\F_q[X]$, and
by $L^*(s, \chi)$ the completed $L$-function that includes the place at infinity. For a Dirichlet character modulo a monic polynomial $v$, we have that
$$L^*(s, \chi) = (1-q^{-s})^{-\lambda_v} L(s, \chi),$$
where $\lambda_v$ is 1 if $\deg v \equiv 0 \pmod \ell$, and 0 otherwise.

Then, for $\chi$ nontrivial, we remark that both $L(s, \chi)$ and $L^*(s, \chi)$ are analytic and non-zero for $\re(s) > 1/2$.

By $\ell$-power reciprocity, we can write this character as
\begin{eqnarray} \label{reciprocity}
\chi_{v, \ell}(v_0)= \sigma\circ \left( \frac{v_0}{v} \right)_\ell = \sigma\left( \left( (-1)^{(q-1)/\ell} \right)^{\deg{v_0} \deg{v}}  \left( \frac{v}{v_0} \right)_\ell\right) = \Psi_{v_0, \ell}(v) \chi_{v_0, \ell}(v),\end{eqnarray}
where $\chi_{v_0,\ell}(v)$ is the Dirichlet character modulo $v_0$ defined above, and $\Psi_{v_0, \ell}(v)$ depends only on the degree of $v$.

If $v=v_\infty$, let $a_n$ be the principal coefficient of $f$. Then we
define
\[\chi_{v_\infty, \ell}(f) :=\left\{\begin{array}{cc}\sigma(a_n) & \deg f \equiv 0 \pmod \ell,\\ 0 & \deg f \not \equiv 0 \pmod \ell. \end{array}\right.\]
We note that the above definition together with \eqref{chi(vinfty)} agree with $\ell$-power reciprocity in the following way
\begin{equation}\label{eq:crazyreciprocitygen}
\chi_{v,\ell}(v_\infty)=\left((-1)^{(q-1)/\ell}\right)^{\deg v} \chi_{v_\infty,\ell}(v)=\left\{\begin{array}{cc}1 & \deg v \equiv 0 \pmod \ell,\\ 0 & \deg v \not \equiv 0 \pmod \ell. \end{array}\right.
\end{equation}
We have used that $v$ is a monic polynomial, which implies that $\chi_{v_\infty,\ell}(v)=1$ when $\ell \mid \deg v$, that $\deg v_\infty=1$, and that $\left((-1)^{(q-1)/\ell}\right)^{\deg v}=1$ when $\ell\mid \deg v$ and $q$ odd, and is trivially 1 when $q$  is even since then we have even characteristic.

Finally, we remark that by the above, the Kronecker symbol codifies ramification in extensions in the usual way.
Let $f \in \F_q[X]$ (not necessarily monic). Then,
\[  \chi_{v, \ell}(f) = \begin{cases}
             1 & v \mbox{ splits in } K(\sqrt[\ell]{f}),\\
             \xi_\ell^k, \; \mbox{for some $1 \leq k \leq \ell-1$} & v \mbox{ is inert in } K(\sqrt[\ell]{f}),\\
             0 & v \mbox{ ramifies in } K(\sqrt[\ell]{f}),\\
            \end{cases}\]
            where $\xi_\ell \in \C$ is a primitive $\ell$th root of $1$.

We now proceed to prove convergence results for Dirichlet series and similar functions.

\begin{lem} \label{Lem-Dirichlet}
Let $\chi$ be a nontrivial Dirichlet character and
let $\Psi$ be a function on $\F_q[X]$ such that $\Psi(F)=\Psi(G)$ when $\deg F=\deg
G$. Then
\[
L(s,\Psi \chi)=\sum_{\substack{F \in \F_q[X] \\ F \text{ monic }}} \frac{\Psi(F)\chi(F)}{|F|^s}
\]
is an analytic function on $\mathbb C$.

\end{lem}
\begin{proof}
Let
\[
A(n, \Psi, \chi)=
\sum_{\substack{F \in \F_q[X],\\ F \text{ monic,} \\ \deg F=n}} \Psi(F)\chi(F).
\]
Then $L(s,\Psi \chi)$ equals
\begin{equation}\label{eq:LDirichlet}
\sum_{n = 0}^\infty \frac{A(n, \Psi, \chi)}{q^{ns}}.
\end{equation}
We note that
\[
A(n, \Psi, \chi)=\Psi(G)\sum_{\substack{F \in \F_q[X],\\ F \text{ monic,} \\ \deg F=n}}\chi(F)
\]
for any polynomial $G$ of degree $n$,
and thus $A(n, \Psi, \chi)=0$ if $n$ is greater than or equal to the degree  of the  modulus of $\chi$ by the orthogonality relations of
characters. This implies that the sum in \eqref{eq:LDirichlet} is finite and therefore $L(s, \Psi \chi)$ is analytic.

\end{proof}

\begin{lem} \label{lem-generalM} Let $\xi_\ell$ be a primitive $\ell$th root of 1.  Let ${\mathcal V}_R, {\mathcal V}_S$ and ${\mathcal V}_U$ be finite subsets of places of
${\mathcal V}_K$ such that ${\mathcal V}_S = \left\{ v_1, \dots, v_n \right\} \subset {\mathcal V}_U$, and
${\mathcal V}_U \cap {\mathcal V}_R = \varnothing$.
For each $0 \leq j \leq \ell-1$, and each tuple $(k_1, \dots, k_n) \neq (0, \dots, 0)$
with $0 \leq k_i \leq \ell-1$, let
\begin{eqnarray*}
&& {\mathcal M}_{j, k_1, \dots, k_n}(s; {\mathcal V}_R, {\mathcal V}_S, {\mathcal V}_U) \\ &&:=
\prod_{v \not\in {\mathcal V}_R \cup {\mathcal V}_U} \left(1+\left(\xi_\ell^{j\deg v} \prod_{h=1}^n \chi_{v, \ell}(v_h)^{k_h} +\cdots
+\xi_\ell^{(\ell-1)j\deg v} \prod_{h=1}^n \chi_{v,\ell}(v_h)^{(\ell-1)k_h} \right)Nv^{-(\ell-1)s}\right).
\end{eqnarray*}
Then, each ${\mathcal M}_{j, k_1, \dots, k_n}(s; {\mathcal V}_R, {\mathcal V}_S, {\mathcal V}_U)$  converges absolutely for  $\re(s) > \frac{1}{\ell-1}$ and has analytic continuation to the region $\re(s) > \frac{1}{2(\ell-1)}$.
\end{lem}
In the case where we have only one place $v_0 \in \mathcal{V}_K$ with prescribed ramification $\epsilon_0 \in \{\mathrm{ramified}, \mathrm{split}, \mathrm{inert}\}$, we will denote the above function by
\begin{equation}\label{msimple}{\mathcal M}_{j, k}(s;v_0,\epsilon_0) :=  {\mathcal M}_{j, k_1}(s; {\mathcal V}_R, {\mathcal V}_S, {\mathcal V}_U).
 \end{equation}

\begin{proof} For the absolute convergence, we have that the convergence of
$\prod_{v} (1+(\ell-1)|Nv^{-s(\ell-1)}|)$ is equivalent to that of $\sum_{v}\frac{1}{Nv^{s(\ell-1)}}$ and this convergence follows in the same way
as the absolute convergence for the zeta function $\zeta_K(s)$ in $\re(s)>1$.

For the analytic continuation, we  write
\begin{eqnarray*}
&&{\mathcal M}_{j, k_1, \dots, k_n}(s; {\mathcal V}_R, {\mathcal V}_S, {\mathcal V}_U) \\ &&= \mathcal{C}^{1}_{j, k_1, \dots, k_n}(s)
\; \prod_{i=1}^{\ell-1}
\prod_{v \not\in {\mathcal V}_R \cup {\mathcal V}_U} \left( 1 + \xi_\ell^{i j\deg v} \prod_{h=1}^n \chi_{v, \ell}(v_h)^{i k_h} Nv^{-(\ell-1)s}\right)\\
&&= \mathcal{C}^{2}_{j, k_1, \dots, k_n}(s) \prod_{i=1}^{\ell-1} \prod_{v \not\in {\mathcal V}_R\cup \mathcal{V}_U} \left( 1 - \xi_\ell^{i j \deg v}
\prod_{h=1}^{n} \Psi_{v_h, \ell}(v)^{i k_h} \chi_{v_h,\ell}(v)^{i k_h} Nv^{-(\ell-1)s} \right)^{-1},
\end{eqnarray*}
where we have used $\ell$-power reciprocity \eqref{reciprocity}, and
where $\mathcal{C}^{1}_{j, k_1, \dots, k_n}(s)$ and $\mathcal{C}^{2}_{j, k_1, \dots, k_n}(s)$ are analytic functions for $\re(s) > 1/2(\ell-1)$ as the Euler products converge absolutely in that region.
For each $1 \leq i \leq \ell-1$, each $0 \leq j \leq \ell-1$ and each tuple $(k_1, \dots, k_n)$ as above, we have that the functions
\begin{eqnarray*}
L_{i, j,k_1, \dots k_n}(s) &=& \prod_{v \not\in {\mathcal V}_R\cup \mathcal{V}_U} \left( 1 - \xi_\ell^{i j \deg v}
\prod_{h=1}^{n} \Psi_{v_h, \ell}(v)^{i k_h} \chi_{v_h,\ell}(v)^{i k_h}
 Nv^{-(\ell-1)s} \right)^{-1} \\
&=& L(s_1, \Psi_{i,j,k_1, \dots, k_h} \; \chi_{i,j,k_1, \dots, k_h})
\end{eqnarray*}
are twisted Dirichlet functions as in Lemma \ref{Lem-Dirichlet}, where $s_1=(\ell-1)s$,
\begin{eqnarray*}
\Psi_{i,j,k_1, \dots, k_h}(v) &=& \xi_\ell^{i j \deg v} \prod_{h=1}^{n} \Psi_{v_h, \ell}(v)^{i k_h}, \\
\chi_{i,j,k_1, \dots, k_h}(v) &=& \prod_{h=1}^{n} \chi_{v_h,\ell}(v)^{i k_h}.
\end{eqnarray*}
Then, $\Psi_{i,j,k_1, \dots, k_h}(v)$
depends only on the degree of $v$, and $\chi_{i,j,k_1, \dots, k_h}(v)$ is a nontrivial Dirichlet character since $1 \leq i \leq \ell-1$, $(k_1, \dots, k_n) \neq 0$ and the product is taken over different places so that there is no possibility of cancelation.
Applying Lemma \ref{Lem-Dirichlet}, this completes the proof of the analytic continuation.

\end{proof}

Let $\xi_\ell$ be a primitive $\ell$th root of 1.
We now prove a result bounding the meromorphic continuation of the functions
\begin{eqnarray} \label{def-as}
\mathcal{A}(s)&:=&\prod_{v} \left(1+(\ell-1)Nv^{-(\ell-1)s}\right)\\
\label{def-bs}
\mathcal{B}(s)&:=&\prod_{v} \left(1+(\xi_\ell^{\deg v}+\cdots
+\xi_\ell^{(\ell-1)\deg v})Nv^{-(\ell-1)s}\right)
\end{eqnarray}
on the line ${\re}(s) = \frac{1}{2 (\ell-1)} + \varepsilon$, for any $\varepsilon > 0$.
We remark that the Euler products converge (absolutely and uniformly) for $\re(s) > \frac{1}{\ell-1}$.

Unless otherwise specified, we continue to use the convention that all the sums and products over $v$ include all places, including the place at infinity.

\begin{lem}\label{chantalmagictrivial} Let $0<\varepsilon < \frac{1}{2(\ell-1)}.$ The functions
$\mathcal A (s)$ and $\mathcal{B}(s)$ have meromorphic continuation to the region $\mathrm{Re}(s)>\frac{1}{2(\ell-1)}+\varepsilon$, and their only singularities in this region are poles on the line $\re(s) = \frac{1}{\ell-1}$.
Furthermore, both functions are absolutely bounded on the region $\frac{1}{2(\ell-1)}<{\re}(s) < \frac{1}{\ell-1}-\varepsilon.$
\end{lem}

\begin{proof}
For $\mathrm{Re}(s) > \frac{1}{\ell-1},$  we have
\begin{eqnarray*}
\mathcal{A}(s) & = & \prod_{v} \left(1+(\ell-1)Nv^{-(\ell-1)s}\right)\\
&=& \zeta_K((\ell-1)s)^{\ell-1} \prod_{v} \left(1+(\ell-1)Nv^{-(\ell-1)s}\right) (1-Nv^{-(\ell-1)s})^{\ell -1}\\
& = &  \zeta_K((\ell-1)s)^{\ell-1} \prod_{v} \left(1+(\ell-1)Nv^{-(\ell-1)s}\right) \left(1 - (\ell-1) Nv^{-(\ell-1)s}  \right. \\
&& \;\;\;\;\; + \left. \binom{\ell-1}{2} Nv^{-2(\ell-1)s}
 + Nv^{-3(\ell-1)s} O_\ell\left( 1 \right) \right) \\
&=&   \zeta_K((\ell-1)s)^{\ell-1} \prod_{v}  \left(1 - \binom{\ell-1}{2} Nv^{-2(\ell-1)s} + Nv^{-3(\ell-1)s} O_\ell\left( 1 \right) \right) \\
& = & \mathcal{C}(s) \zeta_K((\ell-1)s)^{\ell-1} \prod_{v} \left(1 - Nv^{-2(\ell-1)s} \right)^{\frac{\ell(\ell-1)}{2}}  \\
& =  & \mathcal{C}(s)  \frac{\zeta_K((\ell-1)s)^{\ell-1}}{\zeta_K(2(\ell-1)s)^{\frac{\ell(\ell-1)}{2}}},
\end{eqnarray*}
where $\mathcal C(s)$ is analytic for $\re(s) > \frac{1}{3(\ell-1)} + \varepsilon$.
Thus for $s= \frac{1}{2(\ell-1)} + \varepsilon,$ as $\varepsilon$ goes to zero  the function $\mathcal A(s)$ converges to zero, and the result follows. The poles
are given by those of $\zeta_K((\ell-1)s)$, namely $s=1/(\ell-1)$, with multiplicity $\ell-1$.

Similarly, for $\mathrm{Re}(s) > \frac{1}{\ell-1},$  we have
\begin{eqnarray*}
\mathcal{B}(s) & = & \prod_{v} \left(1+(\xi_\ell^{\deg v}+\cdots
+\xi_\ell^{(\ell-1)\deg v})Nv^{-(\ell-1)s}\right) \\
&=& \prod_{j=1}^{\ell-1}  Z_K(\xi_\ell^j u) \prod_{v} \left(1+(\xi_\ell^{\deg v}+\cdots
+\xi_\ell^{(\ell-1)\deg v})Nv^{-(\ell-1)s}\right) \prod_{j=1}^{\ell-1} \left(1 - \xi_\ell^{j \deg{v}} Nv^{- (\ell-1)s}
\right),\\
\end{eqnarray*}
where $u=q^{-(\ell-1)s}$ and $Z_K(u):=\frac{1}{(1-qu)(1-u)}$
is the zeta function of $K$.

Thus, we have,
\begin{eqnarray*}
\mathcal{B}(s) &=& \prod_{j=1}^{\ell-1}  Z_K(\xi_\ell^j u) \prod_{v} \left(1+(\xi_\ell^{\deg v}+\cdots
+\xi_\ell^{(\ell-1)\deg v})Nv^{-(\ell-1)s}\right)\\
&\times& \prod_{v} \biggl(1-(\xi_\ell^{\deg v}+\cdots
+\xi_\ell^{(\ell-1)\deg v})Nv^{-(\ell-1)s}   \\
& & + \left(\sum_{1\leq i<j\leq \ell -1}\xi_\ell^{i \deg v} \xi_\ell^{j \deg{v}}\right) Nv^{-2(\ell-1)s} + Nv^{-3(\ell-1)s} O_\ell(1)
\biggr) \\
&=& \mathcal{C}(s)\prod_{j=1}^{\ell-1}  Z_K(\xi_\ell^j u) \prod_{v} \left( 1 + c(\ell) Nv^{-2(\ell-1)s} \right),
\end{eqnarray*}
where
\begin{align*}
c(\ell) &= - \left(\xi_\ell^{\deg v}+\cdots
+\xi_\ell^{(\ell-1)\deg v}\right)^2 + \sum_{1\leq i<j\leq \ell -1} \xi_\ell^{i \deg v} \xi_\ell^{j \deg{v}} \\
&= - \sum_{1\leq i\leq j\leq \ell -1} \xi_\ell^{i \deg v} \xi_\ell^{j \deg{v}}\\
&=\begin{cases}
   -\frac{\ell(\ell-1)}{2} & \ell \mid \deg v, \ell>2,\\
   0 & \ell \nmid \deg v, \ell >2,\\
  -1 & \ell =2,
   \end{cases}
\end{align*}
and $\mathcal C(s)$ is analytic for $\re(s) > \frac{1}{3(\ell-1)} + \varepsilon$.
Thus for $s= \frac{1}{2(\ell-1)} + \varepsilon,$ as $\varepsilon \to 0$,  the function $\mathcal B(s)$ converges to $0$, and the result follows.

The poles are those of $Z_K(\xi_\ell^j u)$, namely, poles of order one at  $s=\frac{1}{\ell-1}+\frac{2j\pi i}{(\ell-1)\ell\log q}$.

\end{proof}

\section{$\ell$-Cyclic Extensions}\label{section:caseell}

In this section, we give the proofs of the main results of this paper. We will continue with the notation introduced in the earlier sections. Recall that for a fixed prime $\ell$, $N(\Z/\ell \Z, \od)$ denotes the number of extensions of $K$ with Galois group $\Z/\ell \Z$ such that the degree of the conductor is $\od$.
As before, $\xi_\ell \in \C$  stands for a primitive $\ell$th root of 1.

We start by proving the first part of Theorem \ref{thm:main}.

\begin{thm} \label{caseL} Let $\ell\in \Z$ be a prime. We have
\begin{eqnarray}\label{ell-denom}
  N(\Z / \ell \Z, \od) &=& C_\ell \; q^\od  P_\ell(\od) + O \left(
  q^{\left( \frac{1}{2} + \varepsilon \right) \od} \right),
  \end{eqnarray}
where $P_\ell(X) \in \R[X]$ is a monic polynomial of degree $\ell-2$, and
where $C_\ell$ is the non-zero constant given by
\begin{eqnarray*} C_\ell =   \frac{(1-q^{-2})^{\ell-1}}{(\ell-2)!}
\prod_{j=1}^{\ell-2}  \prod_v \left(1-\frac{jq^{-2\deg v}}{(1+q^{-\deg
v})(1+jq^{-\deg v})}\right).\end{eqnarray*}
 \end{thm}

\begin{proof} To  compute $N(\Z / \ell \Z, \od)$, we consider the Dirichlet series $\mathcal{F}(s)$, which is the generating function with an added constant, namely,
\begin{eqnarray*}
\mathcal{F}(s)  := \ell +  \sum_{\Gal(L/K) \cong \Z/\ell \Z} \mathfrak{D}(L/K)^{-s}.
\end{eqnarray*}
We claim that
\begin{eqnarray*}
\mathcal{F}(s)
 &=&\sum_{j=0}^{\ell-1}\prod_{v} \left(1+(\xi_\ell^{j\deg v}+\cdots
+\xi_\ell^{(\ell-1)j\deg v})Nv^{-(\ell-1)s}\right)\\
&=&\prod_{v} \left(1+(\ell-1)Nv^{-(\ell-1)s}\right) + (\ell-1) \prod_{v} \left(1+(\xi_\ell^{\deg v}+\cdots
+\xi_\ell^{(\ell-1)\deg v})Nv^{-(\ell-1)s}\right)\\
&=& \mathcal{A}(s)+(\ell-1)\mathcal{B}(s).
\end{eqnarray*}
Indeed,  by Propositions \ref{unram} and \ref{globalccresult} the  extensions $L/K$ are in one-to-one correspondence with the maps $\phi: \pi_\infty ^\Z \times \prod_v \co_v^\times \to  \Z/\ell\Z$  satisfying
\eqref{eq:modell}.
Let $\cond(\phi)$ be the conductor of such a map $\phi$, and $v$ a place of the conductor. In the first line above,
the $i$th term $\xi_\ell^{ij\deg v}Nv^{-(\ell-1)s}$ in each Euler product corresponds to the map where
$\phi_{v}(g_v)=i$  for $1 \leq i \leq \ell-1$. Therefore, considering all the places $v$ of $\cond(\phi)$, the term in the $j$th Dirichlet series above corresponding to the global map
$\phi$ equals
\[
\left(\xi_\ell^{\sum_v j \phi_v(g_v)\deg v} \right)\times N(\cond(\phi))^{-(\ell-1)s}
\]
for $0 \leq j \leq \ell-1$.
Thus the sum of those terms over the index $j$ yields
$\ell N(\cond(\phi))^{-(\ell-1)s}$ if $\sum_v \phi_v(g_v)\deg v \equiv 0\pmod \ell$ and  $0$
otherwise, and we recover \eqref{eq:modell}. Notice that the $\ell$ factor multiplying
$N(\cond(\phi))^{-(\ell-1)s}$ is accounting for  the different extensions with the same conductor $K(\sqrt[\ell]{f}), K(\sqrt[\ell]{\beta f}), \dots, K(\sqrt[\ell]{\beta^{\ell-1}f})$ for $\beta \in \F_q^\times$ not an $\ell$th power.
Similarly, the constant $\ell$ in the definition of $\mathcal{F}(s)$ accounts for the extensions  $ K(\sqrt[\ell]{\beta}), \dots, K(\sqrt[\ell]{\beta^{\ell-1}})$ for $\beta \in \F_q^\times$ not an $\ell$th power, as well as the $K$-algebra given by the
completely split cover.

Using the identity
$$
\frac{1 + (\ell-1)u}{(1+u)^{\ell-1}} = \prod_{j=1}^{\ell-2} \left( 1 - \frac{ju^2}{(1+u)(1+ju)} \right),
$$ we write
\begin{align*}
\mathcal{A}(s)&=\prod_{v} \left(1+(\ell-1)Nv^{-(\ell-1)s}\right)\\
&=\left(\frac{\zeta_K((\ell-1)s)}{\zeta_K(2(\ell-1)s)}\right)^{\ell-1}\prod_{j=1
}^{\ell-2} \prod_v
\left(1-\frac{jNv^{-2(\ell-1)s}}{(1+Nv^{-(\ell-1)s})(1+jNv^{-(\ell-1)s})}
\right),\\
\end{align*}
where the absolute convergence of the infinite products for $\re(s)>\frac{1}{2(\ell-1)}$ follows from that of $\sum_v \frac{1}{Nv^{2(\ell-1)s}}$.

We also write
\begin{align*}
\mathcal{B}(s)&=\prod_{v} \left(1+(\xi_\ell^{\deg v}+\cdots
+\xi_\ell^{(\ell-1)\deg v})Nv^{-(\ell-1)s}\right)
\\
&=\prod_v \prod_{j=1}^{\ell-1}\left( 1+\xi_\ell^{j\deg v} Nv^{-(\ell-1)s}\right)
\prod_{v} \frac{ \left(1+(\xi_\ell^{\deg v}+\cdots +\xi_\ell^{(\ell-1)\deg
v})Nv^{-(\ell-1)s}\right)}{\prod_{j=1}^{\ell-1} \left( 1+\xi_\ell^{j\deg v}
Nv^{-(\ell-1)s}\right)},
\end{align*}
where the absolute convergence of the infinite products follows in the same way as the products appearings in $\mathcal{A}(s)$.

Recall from Lemma \ref{chantalmagictrivial} that $\mathcal{A}(s)$ is a meromorphic function on
$\mathrm{Re}(s)>\frac{1}{2(\ell-1)}$ with a pole of order $\ell-1$ at
$s=\frac{1}{\ell-1}$ in the region $B_{\ell-1}$ as defined in \eqref{B_k}. The function $\mathcal{B}(s)$ is also meromorphic in
$\re(s)>\frac{1}{2(\ell-1)}$, with simple poles at
$s_{j}=\frac{1}{\ell-1}+\frac{2j\pi i}{(\ell-1)\ell \log q}$
for $|2j|<\ell$  in the region $B_{\ell-1}$.

We set $u=q^{-(\ell-1)s}$, and write $A(u):=\mathcal{A}(s)$ and
$B(u):=\mathcal{B}(s)$. Thus,
\begin{align*}
 A(u)&=\left(\frac{(1-qu^2)(1+u)}{(1-qu)}\right)^{\ell-1}\prod_{j=1}^{\ell-2}
\prod_v \left(1-\frac{ju^{2\deg v}}{(1+u^{\deg v})(1+ju^{\deg v})}\right),\\
B(u)&= \prod_v \prod_{j=1}^{\ell-1}\left( 1+(\xi_\ell^{j}u)^{\deg v}\right)
\prod_{v} \frac{ \left(1+(\xi_\ell^{\deg v}+\cdots +\xi_\ell^{(\ell-1)\deg
v})u^{\deg v}\right)}{\prod_{j=1}^{\ell-1}\left( 1+(\xi_\ell^{j}u)^{\deg
v}\right)}\\
&=\prod_{j=1}^{\ell-1}\frac{Z_K(\xi_\ell^j u)}{Z_K(\xi_\ell^{2j}
u^2)}\prod_v\frac{ \left(1+b(v)u^{\deg v}\right)}{\prod_{j=1}^{\ell-1}\left( 1+(\xi_\ell^{j}u)^{\deg
v}\right)},\\
\end{align*}
 where $
Z_K(u)=\frac{1}{(1-qu)(1-u)}$ and $b(v)=\xi_\ell^{\deg v}+\cdots +\xi_\ell^{(\ell-1)\deg
v}.$

Fix any $\delta$ with $\frac{1}{2(\ell-1)}<\delta<\frac{1}{\ell-1}$.
Then $A(u)$ and $B(u)$ are meromorphic functions on the disk $\{u :  |u|\leq
q^{-\delta}\}$. We see that $A(u)$ has a pole of order $\ell-1$ at $u=1/q$ and $B(u)$ has
$(\ell-1)$ simple poles at $u=(q \xi_\ell^j )^{-1}$ for $j=1,\cdots, \ell-1$. Then,
applying Theorem \ref{thm:Tauberian} and Lemma \ref{chantalmagictrivial} to $\mathcal F(s)=\mathcal A(s)+(\ell-1)\mathcal B(s)$
with $\delta=\frac{1}{2(\ell-1)}+\varepsilon$ for $\varepsilon > 0$,  we have
that
\begin{eqnarray}  \label{plughere}
N(\Z / \ell \Z, \od) = - \mathrm{Res}_{u=q^{-1}} \frac{A(u)}{u^{\od+1}} -
\sum_{j=1}^{\ell-1}
\mathrm{Res}_{u=(q\xi_\ell^{j})^{-1}}  \frac{B(u)}{u^{\od+1}} + O \left(
q^{(1/2 + \varepsilon) \od} \right) .
\end{eqnarray}

We compute,
\begin{align*}
&\mathrm{Res}_{u=q^{-1}} \frac{A(u)}{u^{\od+1}}
\\&=\lim_{u \to q^{-1}} \frac{1}{(\ell-2)!}\frac{d^{\ell-2}}{d u^{\ell-2}}
(u-q^{-1})^{\ell-1}
\frac{1}{u^{\od+1}}\left(\frac{(1-qu^2)(1+u)}{(1-qu)}\right)^{\ell-1}\prod_{j=1}
^{\ell-2} \prod_v \left(1-\frac{ju^{2\deg v}}{(1+u^{\deg v})(1+ju^{\deg
v})}\right)\\
&=\lim_{u \to q^{-1}} \frac{1}{(\ell-2)!}\frac{d^{\ell-2}}{d u^{\ell-2}}
\left(\frac{(-(1-qu^2)(1+u))^{\ell-1}}{q^{\ell-1}u^{\od+1}}\right)\prod_{j=1}^{
\ell-2} \prod_v \left(1-\frac{ju^{2\deg v}}{(1+u^{\deg v})(1+ju^{\deg
v})}\right).\\
\end{align*}
Let
\begin{equation}\label{eq:H}
H_\ell(u):=\frac{1}{(\ell-2)!}\left(\frac{(-(1-qu^2)(1+u))^{\ell-1}}{q^{\ell-1}}\right)\prod_{j=1}^{
\ell-2} \prod_v \left(1-\frac{ju^{2\deg v}}{(1+u^{\deg v})(1+ju^{\deg
v})}\right).
\end{equation}
Then, using the product rule for derivatives, we get
\begin{align*}
 \mathrm{Res}_{u=q^{-1}} \frac{A(u)}{u^{\od+1}}&=\lim_{u \to q^{-1}} \sum_{i=0}^{\ell-2} \binom{\ell-2}{i} \frac{d^i}{d u^i}\left(\frac{1}{u^{\od+1}}\right) \frac{d^{\ell-2-i}}{d u^{\ell-2-i}} H_\ell(u)\\
&=\lim_{u \to q^{-1}} \sum_{i=0}^{\ell-2} \binom{\ell-2}{i}\frac{(-1)^i(\od+1)\cdots(\od+i)}{u^{\od+i+1}} \frac{d^{\ell-2-i}}{d u^{\ell-2-i}} H_\ell(u)\\
&=\sum_{i=0}^{\ell-2} \binom{\ell-2}{i}(-1)^i(\od+1)\cdots(\od+i)q^{\od+i+1} \left.\frac{d^{\ell-2-i}}{d u^{\ell-2-i}} H_\ell(u)\right|_{u=q^{-1}},
\end{align*}
which proves that this residue is given by a polynomial  in $\od$.

 We take a closer look at the main term of this polynomial, which is the dominating term when $\od \to \infty$. We obtain
 \begin{align*}
 \mathrm{Res}_{u=q^{-1}} \frac{A(u)}{u^{\od+1}}&=\lim_{u \to q^{-1}} \frac{1}{(\ell-2)!}\frac{(-1)^{\ell-2}(\od+1)\cdots
(\od+\ell-2)}{u^{\od+\ell-1}}
\left(\frac{(-(1-qu^2)(1+u))^{\ell-1}}{q^{\ell-1}}\right)
\\ & \times \prod_{j=1}^{\ell-2} \prod_v \left(1-\frac{ju^{2\deg v}}{(1+u^{\deg
v})(1+ju^{\deg v})}\right)(1+O(1/\od))\\
&= -\frac{\od^{\ell-2}}{(\ell-2)!} (1-q^{-2})^{\ell-1}q^\od\prod_{j=1}^{\ell-2}
\prod_v \left(1-\frac{jq^{-2\deg v}}{(1+q^{-\deg v})(1+jq^{-\deg
v})}\right)(1+O(1/\od)).
\end{align*}

For the other residues, coming from simple poles,
\begin{align*}
& \mathrm{Res}_{u=(q\xi_\ell^{j_0})^{-1}} \frac{B(u)}{u^{\od+1}}
 \\
&=\lim_{u \to q^{-1}\xi_\ell^{-j_0}}
\frac{(u-q^{-1}\xi_\ell^{-j_0})}{u^{\od+1}}
\prod_{j=1}^{\ell-1}\frac{(1-q\xi_\ell^{2j} u^2)(1+\xi_\ell^j u )}
{(1-q\xi_\ell^j u)}\prod_v\frac{ \left(1+b(v)u^{\deg v}\right)}{\prod_{j=1}^{\ell-1}\left(
1+(\xi_\ell^{j}u)^{\deg v}\right)}\\
&=\lim_{u \to q^{-1}\xi_\ell^{-j_0}}
\frac{-(1-q\xi_\ell^{2j_0} u^2)(1+\xi_\ell^{j_0} u )}{u^{\od+1}q\xi_\ell^{j_0}}
\prod_{j=1, j\not = j_0}^{\ell-1}\frac{(1-q\xi_\ell^{2j} u^2)(1+\xi_\ell^j u )}
{(1-q\xi_\ell^j u)}\prod_v\frac{ \left(1+b(v)u^{\deg v}\right)}{\prod_{j=1}^{\ell-1}\left(
1+(\xi_\ell^{j}u)^{\deg v}\right)}\\
&=
-(q\xi_\ell^{j_0})^{\od}(1-q^{-2}) \prod_{j=1, j\not =
j_0}^{\ell-1}\frac{(1-q^{-1}\xi_\ell^{2j-2j_0})(1+q^{-1}\xi_\ell^{j-j_0} )}
{(1-\xi_\ell^{j-j_0})}\prod_v\frac{ \left(1+b(v)(q^{-1}\xi_\ell^{-j_0})^{\deg
v}\right)}{\prod_{j=1}^{\ell-1}\left( 1+(q^{-1}\xi_\ell^{j-j_0})^{\deg
v}\right)}.
\end{align*}
We note that the line above is $O(q^\od)$ and it contributes to the constant coefficient of $P_\ell(\od)$.

Replacing the residues in \eqref{plughere} with the equations above completes the proof.    \end{proof}

In spite of the fact that Corollary \ref{thm:specialcase} can be deduced from the statement of Theorem \ref{thm:main}, we will prove it first and independently of Theorem \ref{thm:main} as a way of introducing the key ideas in the
proof of Theorem \ref{thm:main}. The case of $v$ ramified and $\ell=2$ will be discussed later, in Section \ref{section:case2}.

Recall that
\begin{eqnarray}
 C_\ell =   \frac{(1-q^{-2})^{\ell-1}}{(\ell-2)!}
\prod_{j=1}^{\ell-2}  \prod_{v\in \mathcal{V}_K} \left(1-\frac{jq^{-2\deg v}}{(1+q^{-\deg
v})(1+jq^{-\deg v})}\right).\end{eqnarray}
\begin{proof} [Proof of Corollary  \ref{thm:specialcase}]  Since $v_0$ is ramified at a cover $L/K$ if and only if $v_0$ divides $\mathrm{Disc}(L/K)$, the generating function for the number of extensions counted by  $N(\Z / \ell \Z,
\od)$ that are ramified at $v_0$ is
\begin{eqnarray*} \mathcal{F}_R(s)&=&
\sum_{\substack{{\Gal(L/K)\cong \Z/\ell \Z}\\{v_0 \text{ ramified }}}} \mathfrak{D}(L/K)^{-s} \\
&=&(\ell-1)Nv_0^{-(\ell-1)s}\prod_{v\not = v_0}
\left(1+(\ell-1)Nv^{-(\ell-1)s}\right)\\&&+(\ell-1)b(v_0)Nv_0^{-(\ell-1)s}
\prod_{v\not=v_0} \left(1+b(v)Nv^{-(\ell-1)s}\right)\\
&=&\frac{(\ell-1)Nv_0^{-(\ell-1)s}}{1+(\ell-1)Nv_0^{-(\ell-1)s}}\mathcal{A}(s)+(\ell-1)\frac{b(v_0)Nv_0^{-(\ell-1)s}}{1+b(v_0)Nv_0^{-(\ell-1)s}} \mathcal{B}(s)
\end{eqnarray*}
where we have excluded the case   $\phi_{v_0}(g_{v_0})=0$ to account for $v_0$ ramified as stated in Proposition \ref{localmaps}.

With the change of variable $u=q^{-(\ell-1)s}$, we obtain
\[ F_R(u) =\frac{(\ell-1)u^{\deg v_0}}{1+(\ell-1)u^{\deg v_0}}
A(u)+(\ell-1)\frac{b(v_0)u^{\deg v_0}}{1+b(v_0)u^{\deg v_0}}B(u).\]

Then, applying Theorem \ref{thm:Tauberian} and Lemma \ref{chantalmagictrivial} with $\delta = \frac{1}{2 (\ell-1)} + \varepsilon$ for any $\varepsilon > 0$, we get
\begin{eqnarray*}
N(\Z/\ell\Z, \od, v_0, \mbox{ramified}) &=& - \mathrm{Res}_{u=q^{-1}}
\frac{(\ell-1)u^{\deg v_0}}{1+(\ell-1)u^{\deg v_0}}
\frac{A(u)}{u^{\od+1}} \\ && - (\ell-1) \sum_{j=1}^{\ell-1}
\mathrm{Res}_{u=(q\xi_\ell^{j})^{-1}} \frac{b(v_0)u^{\deg v_0}}{1+b(v_0)u^{\deg v_0}} \frac{B(u)}{u^{\od+1}} \\
&&+ O \left(
q^{\left(\frac{1}{2} + \varepsilon\right) \od} \right).\\
\end{eqnarray*}
For the residue involving the function $A(u)$, we have
\begin{equation*}
\mathrm{Res}_{u=q^{-1}} \frac{(\ell-1)u^{\deg v_0}}{1+(\ell-1)u^{\deg v_0}}\frac{A(u)}{u^{\od+1}}=\lim_{u \to q^{-1}} \frac{d^{\ell-2}}{d u^{\ell-2}} \frac{(\ell-1)u^{\deg v_0}}{1+(\ell-1)u^{\deg v_0}} \frac{H_\ell(u)}{u^{\od+1}},
\end{equation*}
where $H_\ell(u)$ is given by \eqref{eq:H}. This yields
\begin{align*}
&\mathrm{Res}_{u=q^{-1}} \frac{(\ell-1)u^{\deg v_0}}{1+(\ell-1)u^{\deg v_0}}\frac{A(u)}{u^{\od+1}}\\
&=\sum_{i=0}^{\ell-2} \binom{\ell-2}{i}  {(-1)^i(\od+1)\cdots(\od+i)}{q^{\od+i+1}}  \left.\frac{d^{\ell-2-i}}{d u^{\ell-2-i}} \frac{(\ell-1)u^{\deg v_0}}{1+(\ell-1)u^{\deg v_0}}H_\ell(u)\right|_{u=q^{-1}},\\
\end{align*}
and we obtain the polynomial in $\od$ as in the case of the proof of Theorem \ref{caseL}. As before, we record the main coefficient as the term dominating
when $\od \rightarrow \infty$ to be
\begin{align*}
& \mathrm{Res}_{u=q^{-1}} \frac{(\ell-1)u^{\deg v_0}}{1+(\ell-1)u^{\deg v_0}}\frac{A(u)}{u^{\od+1}}\\
&= -\frac{\od^{\ell-2}}{(\ell-2)!} (1-q^{-2})^{\ell-1}q^\od \frac{(\ell-1)q^{-\deg v_0}}{1+(\ell-1)q^{-\deg v_0}}
\prod_{j=1}^{\ell-2}
\prod_v \left(1-\frac{jq^{-2\deg v}}{(1+q^{-\deg v})(1+jq^{-\deg
v})}\right)(1+O(1/\od))\\
&= - C_\ell  \frac{(\ell-1) q^{-\deg{v_0}}}{1 + (\ell-1) q^{-\deg{v_q}}} q^\od \od^{\ell-2} (1+O(1/\od)).
\end{align*}

For the residues involving the function $B(u)$, we notice that, since the poles are of order one,
\begin{equation*}\mathrm{Res}_{u=(q\xi_\ell^{j})^{-1}} \frac{b(v_0)u^{\deg v_0}}{1+b(v_0)u^{\deg v_0}} \frac{B(u)}{u^{\od+1}}= \frac{b(v_0)(q\xi_\ell^{j})^{-\deg v_0}}{1+b(v_0)(q\xi_\ell^{j})^{-\deg v_0}}\mathrm{Res}_{u=(q\xi_\ell^{j})^{-1}} \frac{B(u)}{u^{\od+1}}.
\end{equation*}
The number above is equal to $O(q^\od)$ and it will contribute to the constant coefficient of the polynomial $P_R(\od)$.
This proves the result for the number of extensions ramifying at $v_0$. We now consider the case of extensions splitting at $v_0$.
First, we write the generating function for the number of extensions of $K$ unramified at $v_0$ as \begin{align*} \mathcal F_{U}(s)  & = \ell +
\sum_{\substack{{\Gal(L/K)\cong \Z/\ell \Z}\\{v_0 \text{ unramified }}}} \mathfrak{D}(L/K)^{-s} \\
&= \sum_{j=0}^{\ell-1}\prod_{v\not = v_0}\left(1+\left(\xi_\ell^{j\deg v}+\cdots+\xi_\ell^{(\ell-1)j\deg v}\right)Nv^{-(\ell-1)s}\right)\\
&=
\prod_{v\not = v_0}
\left(1+(\ell-1)Nv^{-(\ell-1)s}\right) +(\ell-1)\prod_{v\not=v_0}
\left(1+b(v)Nv^{-(\ell-1)s}\right)\\
&=\frac{1}{1+(\ell-1)Nv_0^{-(\ell-1)s}}\mathcal{A}(s)+\frac{(\ell-1)}{1+b(v_0)Nv_0^{-(\ell-1)s}}\mathcal{B}(s).
\end{align*}

Using the notation of Section \ref{characters}, recall that $b_\ell=\mu^{\frac{q-1}{\ell}}$ where $\mu$ is a generator of $\F_q^\times$
(hence $b_\ell$ is an $\ell$th root of unity in $\F_q^\times$), and $\sigma:\F_q^\times \rightarrow \C$ is a character of order $\ell$.
Let $\rho_\ell=\sigma(b_\ell)$,
which is then  a primitive $\ell$th root of unity in $\mathbb C$. For each $v\not = v_0, v_\infty$, denote by $n_v$ a positive integer such that the image of ${v_0}$ in $\left(\mathcal O_v / (\pi_v) \right)^\times$  is $g_v^{n_v}$. Then $\phi_v(v_0)=n_v\phi_v(g_v)$. Hence by Proposition \ref{localmaps} $v_0$  is unramified and split if and only if $\phi_{v_0}(\mathcal O_{v_0}^\times)=0$ and
\[
-(\deg v_0)\psi_\infty(\pi_\infty)+\sum_{v\neq v_0, v_\infty}n_v\phi_v(g_v)\equiv 0 \pmod \ell,
\]
which is equivalent to
\[
\rho_\ell^{-(\deg v_0)\psi_\infty(\pi_\infty)}\prod_{v\neq v_0, v_\infty}\rho_\ell^{n_v\phi_v(g_v)}=1
\]
for the  primitive $\ell$th root of unity $\rho_\ell$ coming from the choice of primitive root $b_\ell \in \F_q^\times$ that we fixed in Section \ref{characters}.

Thus $v_0\not=v_\infty$ is unramified and split if and only if
$\phi_{v_0}(\mathcal O_{v_0}^\times)=0$ and
\begin{align}\label{eqn:split}
D(v_0) := \rho_\ell^{-(\deg v_0)\psi_\infty(\pi_\infty)}\prod_{v\not =  v_0, v_\infty }\chi_{v,\ell}({v_0})^{\phi_v(g_v)}=1.
\end{align}
Since $D(v_0)$ is a $\ell$th root of unity, we can rewrite \eqref{eqn:split} as
\begin{align}\label{eqn:split-2}
\frac{1}{\ell} \sum_{j=0}^{\ell-1} D(v_0)^j = \begin{cases} 1 & \mbox{if $v_0$ is unramified and split,} \\
0 & \mbox{otherwise,} \end{cases}
\end{align}
and this is the criterion that we will use in the generating series.

Analogously, we also have that $v_\infty$ is unramified and split if and only if $\phi_{v_\infty}(\mathcal O_{v_\infty}^\times)=0$ and
\[\rho_\ell^{-(\deg v_\infty) \psi_\infty(\pi_\infty)}=1,\]
since $\deg v_\infty=1$.

We claim that the Dirichlet series for cyclic extensions splitting at a fixed place $v_0\not = v_\infty$ is
\begin{eqnarray*}
\mathcal{F}_S(s)
 &=& \frac{1}{\ell^2} \sum_{j=0}^{\ell-1} \sum_{k=0}^{\ell-1} \sum_{r=0}^{\ell - 1} \rho_\ell^{-rk \deg{v_0}} \\
&& \times \prod_{v \neq v_0,v_\infty} \left(1+(\xi_\ell^{j\deg v} \chi_{v, \ell}(v_0)^k +\cdots
+\xi_\ell^{(\ell-1)j\deg v} \chi_{v,\ell}(v_0)^{(\ell-1)k} )Nv^{-(\ell-1)s}\right)\\
&&\times \left(1+(\xi_\ell^{j\deg v_\infty}  +\cdots
+\xi_\ell^{(\ell-1)j\deg v_\infty} )Nv_\infty^{-(\ell-1)s}\right).\\
\end{eqnarray*}
Recall  by Propositions \ref{unram}, \ref{globalccresult} and \ref{localmaps} the cyclic  extensions splitting at a fixed place $v_0\neq v_\infty$ are in one-to-one correspondence with the maps $\phi: \pi_\infty ^\Z \times \prod_v \co_v^\times \to  \Z/\ell\Z$  satisfying
\eqref{eq:modell}, together with  the splitting conditions \eqref{eqn:split} and $\phi_{v_0}(\mathcal O_{v_0}^\times)=0$.
Let $\cond(\phi)$ be the conductor of such a map $\phi$, and $v$ a place of the conductor. For each fixed $j,k,r$ in the first line above,
the $i$th term $\rho_\ell^{-rk \deg{v_0}} \xi_\ell^{ij\deg v}  \chi_{v,\ell}(v_0)^{ik} Nv^{-(\ell-1)s}$ in the Euler product corresponds to the map where
$\phi_{v}(g_v)=i$   and $\psi_\infty(\pi_\infty) = r$, for $1 \leq i \leq \ell-1$.  Considering all the places $v$ of $\cond(\phi)$ (including $v_\infty$, which is accounted for in the last line of the equation),
the term in the $j,k,r$th Dirichlet series above corresponding to the global map
$\phi$ equals
\[
\left(\xi_\ell^{\sum_v j \phi_v(g_v)\deg v} \right) \times \rho_\ell^{-rk \deg{v_0}} \prod_{v\not = v_0, v_\infty} \chi_{v, \ell}(v_0)^{k \phi_v(g_v)} \times N(\cond(\phi))^{-(\ell-1)s}.
\]
Summing over $j$, we obtain zero unless condition \eqref{eq:modell} is satisfied. Summing over $r$ covers all the possible values of $\psi_\infty(\pi_\infty)$.
Finally, summing over $k$ yields zero unless condition \eqref{eqn:split} is satisfied. Thus the sum of those terms over $k,j$, together with the correcting factor $\frac{1}{\ell^2}$
will yield
$N(\cond(\phi))^{-(\ell-1)s}$ if both conditions \eqref{eq:modell} and \eqref{eqn:split} are satisfied and zero otherwise.

We also remark that the constant term of $\mathcal{F}_S(s)$ is $\ell$ if $\ell \mid \deg{v_0}$ and $1$ otherwise.

When $v_0=v_\infty$, we have,
\begin{eqnarray*}
\mathcal{F}_S(s)
 &=& \frac{1}{\ell^2} \sum_{j=0}^{\ell-1} \sum_{k=0}^{\ell-1} \sum_{r=0}^{\ell - 1} \rho_\ell^{-rk \deg{v_\infty}} \prod_{v \neq v_\infty} \left(1+(\xi_\ell^{j\deg v}  +\cdots
+\xi_\ell^{(\ell-1)j\deg v} )Nv^{-(\ell-1)s}\right)\\
&=&\frac{1}{\ell} \sum_{j=0}^{\ell-1} \prod_{v \neq v_\infty} \left(1+(\xi_\ell^{j\deg v}  +\cdots
+\xi_\ell^{(\ell-1)j\deg v} )Nv^{-(\ell-1)s}\right)
= \frac{1}{\ell} \mathcal{F}_U(s).
\end{eqnarray*}

By considering the definitions of $\chi_{v_\infty}(v)$ and $\chi_{v}(v_\infty)$, the previous two formulas can both be  written as
\begin{eqnarray*}\label{F-split}
\mathcal{F}_S(s)
 &=& \frac{1}{\ell^2} \sum_{j=0}^{\ell-1} \sum_{k=0}^{\ell-1} \sum_{r=0}^{\ell - 1} \rho_\ell^{-rk \deg{v_0}} \\
 && \times \prod_{v \neq v_0} \left(1+(\xi_\ell^{j\deg v} \chi_{v, \ell}(v_0)^k +\cdots
+\xi_\ell^{(\ell-1)j\deg v} \chi_{v,\ell}(v_0)^{(\ell-1)k} )Nv^{-(\ell-1)s}\right),\\
\end{eqnarray*}
which is valid for any place $v_0$.

Separating the term with $k=0$ from the terms with $k\not = 0$, we obtain,
\begin{eqnarray*}
\mathcal{F}_S(s)
 &=& \frac{1}{\ell} \mathcal{F}_U(s)\\
 &&+\frac{1}{\ell^2} \sum_{j=0}^{\ell-1} \sum_{k=1}^{\ell-1} \left(\sum_{r=0}^{\ell - 1} \rho_\ell^{-rk \deg{v_0}}\right) \mathcal{M}_{j,k} (s,v_0,\mathrm{split}),\\
\end{eqnarray*}
where $\mathcal{M}_{j,k} (s,v_0,\mathrm{split})$ is given by \eqref{msimple}.

Applying
Theorem \ref{thm:Tauberian} and Lemmas \ref{lem-generalM} and \ref{chantalmagictrivial} to the generating function $\mathcal{F}_S(s)$, we get
\begin{eqnarray*}N(\Z/\ell\Z, \od, v_0, \text{split}) &=&
-\frac{1}{\ell} \mathrm{Res}_{u=q^{-1}}\frac{1}{1+(\ell-1)u^{\deg v_0}}\frac{A(u)}{u^{\od+1}}\\
&&- \frac{\ell-1}{\ell} \sum_{j=1}^{\ell-1}\mathrm{Res}_{u=(\xi_\ell^jq)^{-1}}\frac{1}{\ell(1+b(v_0)u^{\deg
v_0})}\frac{B(u)}{u^{\od+1}} \\
&&+ O \left(
q^{(1/2 + \varepsilon) \od}  \right).\\
\end{eqnarray*}

As before, the residue involving the function $A(u)$ yields $q^\od$ times a polynomial in $\od$ of degree $\ell-2$. The main term
when $\od$ goes infinity is given by the leading term of the polynomial, and is
\begin{align*}
 \mathrm{Res}_{u=q^{-1}} \frac{1}{\ell(1+(\ell-1)u^{\deg v_0})}\frac{A(u)}{u^{\od+1}} = - C_\ell q^\od \od^{\ell-2} \frac{1}{\ell(1+(\ell-1)q^{-\deg v_0})}.
\end{align*}

Similarly the value of
\[\mathrm{Res}_{u=(\xi_\ell^jq)^{-1}}\frac{1}{\ell(1+b(v_0)u^{\deg
v_0})}\frac{B(u)}{u^{\od+1}} \]
is $O(q^\od)$ and it contributes to the constant coefficient of $P_S(\od)$.

We now consider the Dirichlet series for cyclic extensions for which a fixed place $v_0$ is inert. It is given by
\begin{eqnarray*}
\mathcal{F}_I(s)
 &=& \mathcal{F}_U(s)-\mathcal{F}_S(s)\\
 &=&\sum_{j=0}^{\ell-1}\prod_{v\not = v_0}\left(1+\left(\xi_\ell^{j\deg v}+\cdots+\xi_\ell^{(\ell-1)j\deg v}\right)Nv^{-(\ell-1)s}\right)\\
 && \hspace{-1.5cm} -\frac{1}{\ell^2} \sum_{j=0}^{\ell-1} \sum_{k=0}^{\ell-1} \sum_{r=0}^{\ell - 1} \rho_\ell^{-rk \deg{v_0}} \prod_{v \neq v_0} \left(1+(\xi_\ell^{j\deg v} \chi_{v, \ell}(v_0)^k +\cdots
+\xi_\ell^{(\ell-1)j\deg v} \chi_{v,\ell}(v_0)^{(\ell-1)k} )Nv^{-(\ell-1)s}\right)\\
&=&\frac{(\ell-1)}{\ell}\sum_{j=0}^{\ell-1}\prod_{v\not = v_0}\left(1+\left(\xi_\ell^{j\deg v}+\cdots+\xi_\ell^{(\ell-1)j\deg v}\right)Nv^{-(\ell-1)s}\right)\\
 && \hspace{-1.5cm} -\frac{1}{\ell^2} \sum_{j=0}^{\ell-1} \sum_{k=1}^{\ell-1} \sum_{r=0}^{\ell - 1} \rho_\ell^{-rk \deg{v_0}} \prod_{v \neq v_0} \left(1+(\xi_\ell^{j\deg v} \chi_{v, \ell}(v_0)^k +\cdots
+\xi_\ell^{(\ell-1)j\deg v} \chi_{v,\ell}(v_0)^{(\ell-1)k} )Nv^{-(\ell-1)s}\right).\\
\end{eqnarray*}

The main term is given by
\begin{align*} \frac{(\ell-1)}{\ell}\mathcal F_{U}(s)
&=\frac{(\ell-1)}{\ell}\left(\frac{1}{1+\left(\ell-1\right) Nv_0^{-(\ell-1)s}}\mathcal{A}(s)+\frac{\ell-1}{1+b(v_0) Nv_0^{-(\ell-1)s}}\mathcal{B}(s)\right).\\
\end{align*}

The proof proceeds exactly as in the split case.  In particular, this proves that
\[N(\Z/\ell\Z, \od, v_0, \text{inert})=(\ell-1)N(\Z/\ell\Z, \od, v_0, \text{split})+ O \left(
q^{(1/2 + \varepsilon)  \od} \right).\]
This concludes the proof of  Corollary \ref{thm:specialcase}.
\end{proof}

We are now ready to prove the main result.
\begin{thm} \label{thm-general-statement}
Let $\mathcal{V}_R, \mathcal{V}_S, \mathcal{V}_I$ be three finite and disjoint sets of places of $K$.
Let $$N(\Z/\ell \Z , \od; \mathcal{V}_R, \mathcal{V}_S,  \mathcal{V}_I)$$ be the number of extensions of $\F_q(X)$ with Galois group $\Z/\ell\Z$ such that
the degree of the conductor is $\od$, and that are ramified at the places of  $\mathcal{V}_R$, (completely) split at the places of
 $\mathcal{V}_S$ and inert at the places of  $\mathcal{V}_I$. Let $\mathcal{V} =
 \mathcal{V}_R  \cup \mathcal{V}_S  \cup \mathcal{V}_I$.
 Then,
\begin{eqnarray*}
N(\Z/\ell \Z , \od; \mathcal{V}_R, \mathcal{V}_S,  \mathcal{V}_I) &=& C_\ell \left(\prod_{v \in
\mathcal V} c_v\right) \; q^\od P_{\mathcal{V}_R, \mathcal{V}_S,  \mathcal{V}_I}(\od)  +  O \left( q^{\left( \frac{1}{2} + \varepsilon \right) \od}
\right),
\end{eqnarray*}
{and}
\begin{eqnarray*}
\frac{N(\Z/\ell\Z, \od; \mathcal{V}_R, \mathcal{V}_S,  \mathcal{V}_I)}{N(\Z/\ell\Z, \od)} &=&
\left( \prod_{v \in \mathcal{V}} c_v \right) \left (1+ O \left( \frac{1}{\od} \right) \right),
\end{eqnarray*}
where $P_{\mathcal{V}_R, \mathcal{V}_S,  \mathcal{V}_I}(X)  \in \R[X]$ is a monic polynomial of degree
$\ell-2$ and $C_\ell$ is  given by
\begin{eqnarray*}
 C_\ell =   \frac{(1-q^{-2})^{\ell-1}}{(\ell-2)!}
\prod_{j=1}^{\ell-2}  \prod_{v\in \mathcal{V}_K} \left(1-\frac{jq^{-2\deg v}}{(1+q^{-\deg
v})(1+jq^{-\deg v})}\right). \end{eqnarray*}
In addition,
$$c_{v} = \begin{cases} \displaystyle  \frac{(\ell-1)q^{-\deg
v}}{1+(\ell-1)q^{-\deg v}} & \mbox{if $v \in {\mathcal V}_R$, }\\
\displaystyle \frac{1}{\ell(1+(\ell-1)q^{-\deg{v}})} & \mbox{if $v \in {\mathcal V}_S$, }\\
\displaystyle \frac{\ell-1}{\ell(1+(\ell-1)q^{-\deg{v}})} & \mbox{if
$v  \in {\mathcal V}_I$. }\end{cases}$$
 \end{thm}
\begin{proof}
Let $\mathcal{V}_U=\mathcal{V}_S\cup \mathcal{V}_I$.

We first construct the Dirichlet generating series with prescribed conditions for $\mathcal{V}_R$, $\mathcal{V}_U$, and $\mathcal{V}_S=\{v_1,\dots,v_n\} \subset \mathcal{V}_U$. In other words,
for the elements $v\in \mathcal{V}_I$ we will only prescribe that they are in $\mathcal{V}_U$ and we will ignore the inert condition for the moment.
We claim that the generating series is then
\begin{eqnarray*}
&&\mathcal{F}_{\mathcal{V}_R,\mathcal{V}_S \subset \mathcal{V}_U}(s)\\
 &&= \frac{1}{\ell^{n+1}} \sum_{j=0}^{\ell-1} \sum_{k_1=0}^{\ell-1} \dots \sum_{k_n=0}^{\ell-1} \sum_{r=0}^{\ell - 1} \rho_\ell^{-r \sum_{h=1}^{n} k_h \deg{v_h}} \\
&& \times \prod_{v \not\in {\mathcal V}_R \cup {\mathcal V}_U} \left(1+\left(\xi_\ell^{j\deg v} \prod_{h=1}^n \chi_{v, \ell}(v_h)^{k_h} +\cdots
+\xi_\ell^{(\ell-1)j\deg v} \prod_{h=1}^n \chi_{v,\ell}(v_h)^{(\ell-1)k_h} \right)Nv^{-(\ell-1)s}\right)\\
&& \times \prod_{v \in {\mathcal V}_R} \left( \xi_\ell^{j\deg v} +\cdots
+\xi_\ell^{(\ell-1)j\deg v} \right) Nv^{-(\ell-1)s}.
\end{eqnarray*}
Let us prove that the above formula is correct. Recall  by Propositions \ref{unram}, \ref{globalccresult} and \ref{localmaps} the cyclic  extensions
which are ramified at the primes of ${\mathcal V}_R$, unramified at the primes of ${\mathcal V}_U$ and split at the primes of ${\mathcal V}_S$ are in 
one-to-one correspondence with the maps $\phi: \pi_\infty ^\Z \times \prod_v \co_v^\times \to  \Z/\ell\Z$  satisfying
\eqref{eq:modell}, together with  the ramification conditions: $\phi_v$ is nontrivial on $\mathcal{O}_v^\times$ for $v \in \mathcal{V}_R$, $\phi_v(\mathcal{O}_v^\times)=0$ for $v\in \mathcal{V}_U$, 
and the splitting conditions \eqref{eqn:split} for $v \in \mathcal{V}_S$. Thus consider all the possible choices of parameters $\left( \left\{ r_v \right\}, r \right)$ where $\phi$ is
determined by setting  $\phi_v(g_v)=r_v$ and $\psi_\infty(\pi_\infty)=r$ in such a way that 
$\phi$ is ramified at the primes of ${\mathcal V}_R$, unramified at the primes of ${\mathcal V}_U$ and split at the primes of
${\mathcal V}_S$. Therefore we have $0 < r_v \leq \ell-1$ for all $v \in {\mathcal V}_R$, and $r_v=0$ for all primes of ${\mathcal V}_U$.
For each fixed $j, k_1, \dots, k_n$, the map $\phi$ with parameters $\left( \left\{ r_v \right\}, r \right)$ corresponds to the component
\[\left( \prod_{{v \not\in {\mathcal V}_R \cup {\mathcal V}_U}} \rho_\ell^{-r \sum_{h=1}^n k_h \deg{v_h}}
\xi_\ell^{j r_v \deg{v}} \prod_{h=1}^n \chi_{v,\ell}(v_h)^{r_v k_h} \right) \times \left(
\prod_{ {v \in {\mathcal V}_R}} \xi_\ell^{r_v j \deg{v}} \right) \times N(\text{Cond}(\phi))^{-(\ell-1)s}\]
of the Euler product. Summing over all $j, k_1, \dots, k_n$, we obtain that the coefficient of $N(\text{Cond}(\phi))^{-(\ell-1)s}$ is given by
\begin{eqnarray*}
&&\left(\sum_{k_1=0}^{\ell-1} \rho_\ell^{-r k_1\deg{v_1}}   \prod_{v} \chi_{v,\ell}(v_1)^{r_v k_1}\right) \times
\dots \times \left(\sum_{k_n=0}^{\ell-1} \rho_\ell^{-r k_n \deg{v_n}}   \prod_{v} \chi_{v,\ell}(v_n)^{r_v k_n}\right) \\
&&\times
\left(\sum_{j=0}^{\ell-1} \xi_\ell^{j \sum_{v \mid \text{Cond}(\phi)} r_v \deg{v}}\right)\\
&&= \begin{cases} \ell^{n+1} & \mbox{if $\sum_{v \mid \text{Cond}(\phi)} r_v \deg{v} \equiv 0 \pmod \ell$ and $\phi$ is split at $v_1, \dots, v_n$,} \\
0 & \mbox{otherwise}. \end{cases}
\end{eqnarray*}

We now write the
generating series as $\mathcal{F}_{\mathcal{V}_R,\mathcal{V}_S \subset \mathcal{V}_U}(s) = \mathcal{F}^{1}_{\mathcal{V}_R,\mathcal{V}_S \subset \mathcal{V}_U}(s) + \mathcal{F}^{2}_{\mathcal{V}_R,\mathcal{V}_S \subset \mathcal{V}_U}(s)$,
where the first series contributes to the main term and the second to the error term. Taking $(k_1, \dots, k_n) = (0, \dots, 0)$ in $\mathcal{F}_{\mathcal{V}_R,\mathcal{V}_S \subset \mathcal{V}_U}(s)$, we have
\begin{eqnarray*}
\mathcal{F}^{1}_{\mathcal{V}_R,\mathcal{V}_S \subset \mathcal{V}_U}(s)
 &=& \frac{1}{\ell^{n}} \left( \prod_{v \in {\mathcal V}_R} \frac{(\ell-1) Nv^{-(\ell-1)s}}{1 + (\ell-1) Nv^{-(\ell-1)s}}
 \prod_{v \in {\mathcal V}_U}  \frac{1}{1 + (\ell-1) Nv^{-(\ell-1)s}}  {\mathcal A}(s) \right. \\
 && \left. + (\ell-1)
 \prod_{v \in {\mathcal V}_R} \frac{b(v) Nv^{-(\ell-1)s}}{1 + b(v) Nv^{-(\ell-1)s}}
 \prod_{v \in {\mathcal V}_U}  \frac{1}{1 + b(v) Nv^{-(\ell-1)s}}  {\mathcal B}(s)
 \right),
\end{eqnarray*}
where as usual $j=0$ gives the function ${\mathcal A}(s)$ defined by \eqref{def-as} and the other values of $j$ give $\ell-1$ copies of the
function ${\mathcal B}(s)$ defined by \eqref{def-bs}.

Taking $(k_1, \dots, k_n) \neq (0, \dots, 0)$ in $\mathcal{F}_{\mathcal{V}_R,\mathcal{V}_S \subset \mathcal{V}_U}(s)$, we have
\begin{eqnarray*}
\mathcal{F}^{2}_{\mathcal{V}_R,\mathcal{V}_S \subset \mathcal{V}_U}(s)
 &=& \frac{1}{\ell^{n+1}} \sum_{j=0}^{\ell-1} \sum_{{k_1, \dots, k_n = 0}\atop{(k_1, \dots, k_n) \neq (0, \dots, 0)}}^{\ell-1}
 G(s) \; M_{j, k_1, \dots, k_n}(s; {\mathcal V}_R, {\mathcal V}_S, {\mathcal V}_U)
\end{eqnarray*}
where $$G(s) =
\sum_{r = 0}^{\ell-1} \rho_\ell^{-r \sum_{h=1}^n k_h \deg{v_h}} \prod_{v \in {\mathcal V}_R} \left( \xi_\ell^{j\deg v} +\cdots
+\xi_\ell^{(\ell-1)j\deg v} \right) Nv^{-(\ell-1)s}$$
is analytic for all $s \in \C$, and
where for  each fixed vector $(k_1, \dots, k_n) \neq (0, \dots, 0)$, and for each $0 \leq j \leq \ell-1$, we have that
\begin{eqnarray*}
&& {\mathcal M}_{j, k_1, \dots, k_n}(s; {\mathcal V}_R, {\mathcal V}_S, {\mathcal V}_U) \\ &=&
\prod_{v \not\in {\mathcal V}_R \cup {\mathcal V}_U} \left(1+\left(\xi_\ell^{j\deg v} \prod_{h=1}^n \chi_{v, \ell}(v_h)^{k_h} +\cdots
+\xi_\ell^{(\ell-1)j\deg v} \prod_{h=1}^n \chi_{v,\ell}(v_h)^{(\ell-1)k_h} \right)Nv^{-(\ell-1)s}\right)
\end{eqnarray*}
as defined in Lemma \ref{lem-generalM}.

Let
$N'(\Z/\ell \Z, \od; \mathcal{V}_R, \mathcal{V}_S, \mathcal{V}_I)$ be the number of extensions where the degree of the conductor is $\od$ and with the prescribed ramification conditions at the primes of ${\mathcal V}_R$ and ${\mathcal V}_S$, and unramified at the primes of ${\mathcal V}_I$, i.e. the extensions counted by the generating series
$\mathcal{F}_{\mathcal{V}_R,\mathcal{V}_S \subset \mathcal{V}_U}(s)$ above.
By  Theorem \ref{thm:Tauberian}, and Lemmas \ref{lem-generalM} and \ref{chantalmagictrivial},
\begin{eqnarray*}&& N'(\Z/\ell \Z, \od; \mathcal{V}_R, \mathcal{V}_S, \mathcal{V}_I) = \\ && -\frac{1}{\ell^n} \left(
\mathrm{Res}_{u=q^{-1}}
\prod_{v \in {\mathcal V}_R} \frac{(\ell-1) u^{\deg{v}}}{1+(\ell-1)u^{\deg v}}
\prod_{v \in {\mathcal V}_U} \frac{1}{1+(\ell-1)u^{\deg v}}
\frac{A(u)}{u^{\od+1}} \right. \\
&&\left. + (\ell-1)\sum_{j=1}^{\ell-1}\mathrm{Res}_{u=(\xi_\ell^jq)^{-1}}
\prod_{v \in {\mathcal V}_R} \frac{b(v) u^{\deg{v}}}{1+b(v)u^{\deg v}}
\prod_{v \in {\mathcal V}_U} \frac{1}{1+b(v)u^{\deg v}}
\frac{B(u)}{u^{\od+1}} \right) \\
&&+ O\left(
q^{(1/2 + \varepsilon) \od} \right).\\
\end{eqnarray*}

As before, the residue involving the function $A(u)$ yields $q^{\od}$ times a polynomial in $\od$ of degree $\ell-2$, and the residues
of $B(u)$ are $O(q^\od)$, so they contribute to the constant coefficient of the polynomial, and not to the main term.
The main term
when $\od$ tends to infinity is then given by the leading term of the polynomial which is
\begin{align}
\nonumber
&&-\frac{1}{\ell^n} \left(
\mathrm{Res}_{u=q^{-1}}
\prod_{v \in {\mathcal V}_R} \frac{(\ell-1) u^{\deg{v}}}{1+(\ell-1)u^{\deg v}}
\prod_{v \in {\mathcal V}_U} \frac{1}{1+(\ell-1)u^{\deg v}}
\frac{A(u)}{u^{\od+1}} \right) \\ \label{residue-line2}
&&= \frac{1}{\ell^n}
\prod_{v \in {\mathcal V}_R} \frac{(\ell-1) q^{-\deg{v}}}{1+(\ell-1)q^{-\deg v}}
\prod_{v \in {\mathcal V}_U} \frac{1}{1+(\ell-1)q^{-\deg v}} \;\;
C_\ell q^{\od} \od^{\ell-2}.
\end{align}

We now proceed to add the conditions at the primes of $\mathcal{V}_I = \mathcal{V}_U\setminus \mathcal{V}_S$. Using inclusion-exclusion,
it is easy to see that
\begin{eqnarray} \label{inclusion-exclusion}
N(\Z/\ell \Z, \od; \mathcal{V}_R, \mathcal{V}_S, \mathcal{V}_I) = \sum_{\tilde{\mathcal{V}}_I\subset \mathcal{V}_I} {(-1)^{|\tilde{\mathcal{V}}_I|}} N'(\Z/\ell \Z, \od; \mathcal{V}_R, \mathcal{V}_S \cup \tilde{\mathcal{V}}_I, \mathcal{V}_I \setminus
\tilde{\mathcal{V}}_I).
\end{eqnarray}
We can rewrite the above equation in terms of the generating series. Let
$\mathcal{F}_{\mathcal{V}_R,\mathcal{V}_S,\mathcal{V}_I}(s)$ be the generating series for the extensions counted
by $N(\Z/\ell \Z, \od; \mathcal{V}_R, \mathcal{V}_S, \mathcal{V}_I)$.
Then, it follows from \eqref{inclusion-exclusion} that
\begin{eqnarray*}
\mathcal{F}_{\mathcal{V}_R,\mathcal{V}_S,\mathcal{V}_I}(s) &=& \sum_{\tilde{\mathcal{V}}_I\subset \mathcal{V}_I} {(-1)^{|\tilde{\mathcal{V}}_I|}}
{\mathcal F}_{\mathcal{V}_R, \mathcal{V}_S \cup \tilde{\mathcal{V}}_I \subset \mathcal{V}_U}(s)
\\ &=& \sum_{\tilde{\mathcal{V}}_I\subset \mathcal{V}_I} {(-1)^{|\tilde{\mathcal{V}}_I|}} \left(
{\mathcal F}_{\mathcal{V}_R, \mathcal{V}_S \cup \tilde{\mathcal{V}}_I \subset \mathcal{V}_U}^{1}(s) +
{\mathcal F}_{\mathcal{V}_R, \mathcal{V}_S \cup \tilde{\mathcal{V}}_I \subset \mathcal{V}_U}^{2}(s) \right),
\end{eqnarray*}
and the main term will be given by the sum of the poles of the generating series ${\mathcal F}_{\mathcal{V}_R, \mathcal{V}_S \cup \tilde{\mathcal{V}}_I \subset \mathcal{V}_U}^{1}(s)$. Using \eqref{residue-line2}, this is given by
\begin{eqnarray*}
&& C_\ell q^{\od} \od^{\ell-2}
\left( \sum_{\tilde{\mathcal{V}}_I\subset \mathcal{V}_I} \frac{(-1)^{|\tilde{\mathcal{V}}_I|}}{\ell^{|{\mathcal V}_S| \cup |\tilde{\mathcal V}_I}|} \prod_{v \in {\mathcal V}_R} \frac{(\ell-1) q^{-\deg{v}}}{1 + (\ell-1)  q^{-\deg{v}}}
 \prod_{v \in {\mathcal V}_U}  \frac{1}{1 + (\ell-1) q^{-\deg{v}}}  \right) \\
&=& C_\ell q^{\od} \od^{\ell-2} \left( \frac{1}{\ell} \right)^{|\mathcal{V}_S|}  \left( \frac{\ell-1}{\ell} \right)^{|\mathcal{V}_I|}
\prod_{v \in {\mathcal V}_R} \frac{(\ell-1) q^{-\deg{v}}}{1 + (\ell-1) q^{-\deg{v}}}
 \prod_{v \in {\mathcal V}_U}  \frac{1}{1 + (\ell-1) q^{-\deg{v}}} \\
 &=& C_\ell \left( \prod_{v \in  {\mathcal V}_R  \cup {\mathcal V}_S \cup  {\mathcal V}_I} c_v \right) q^{\od} \od^{\ell-2},
\end{eqnarray*}
where the $c_v$ are as in Theorem \ref{thm-general-statement}.

Dividing the last line by \eqref{ell-denom} completes the proof
of the statement.
\end{proof}

\subsection{Quadratic extensions}\label{section:case2}

We now look specifically at the case $\ell =2$ as we obtain the number of quadratic extensions of $K$ with conductor $\od$ with no error term, and the ramified case with a better error term without using the Tauberian theorem. The generating function $\mathcal F$  is
\begin{eqnarray*} \mathcal{F}(s)= 2 +\sum_{\Gal(L/K) \cong \Z/2\Z}
\mathfrak{D}(L/K)^{-s} =\prod_{v} \left(1+Nv^{-s}\right)+ \prod_{v}
\left(1+(-1)^{\deg v}Nv^{-s}\right).\end{eqnarray*}

In this case,
\begin{align*}
\mathcal{A}(s)&=\prod_{v} \left(1+Nv^{-s}\right)=\prod_{v}
\frac{\left(1-Nv^{-2s}\right)}{\left(1-Nv^{-s}\right)}\\
&=\frac{\zeta_K(s)}{\zeta_K(2s)}
=\frac{(1-q^{1-2s})(1+q^{-s})}{1-q^{1-s}}.
\end{align*}
After making the change of variables $u=q^{-s}$, we obtain
\[A(u):=\frac{(1-qu^2)(1+u)}{1-qu}.\]

Analogously,
\[\mathcal{B}(s)=\prod_{v} \left(1+(-1)^{\deg
v}Nv^{-s}\right)=\frac{(1-q^{1-2s})(1-q^{-s})}{1+q^{1-s}}\]
which equals $A(-u)$ after the change of variables $u=q^{-s}$. Then,
\begin{eqnarray*}
F(u) &=& A(u)+ A(-u)\\
&=& (1-q u^2)\left(\frac{1+u}{1-qu}
+\frac{1-u}{1+qu}\right) .\end{eqnarray*}

By identifying the coefficients in the power series expansion in $u$ of the
above rational function for $\od >0$ with the coefficients of
\[2+\sum_{\od=1}^\infty N(\Z/2\Z,\od) u^\od,\]
we finally obtain that
\begin{eqnarray} \nonumber
N(\Z/2\Z,\od) &=& \begin{cases} \left(1+ (-1)^\od\right) (q^\od  - q^{\od-2}) &
\od \geq 3,\\
2q^2 & \od =2, \\
0 & \od =1
\end{cases}   \\ \label{denominator}
&=&  \begin{cases} 2(q^\od - q^{\od-2}) & \od >2, \od \textrm{ even},\\
2q^2 & \od =2, \\
0 & \od \textrm{ odd.} \end{cases}\end{eqnarray}

\begin{rem} Recall that the number of square-free monic polynomials of degree
$d>1$ is $q^d - q^{d-1}$. In this case,  we are counting twice the number
of square-free monic polynomials. The counting
happens twice since every monic square-free polynomial $f$ gives two quadratic
extensions corresponding to $K (\sqrt{f}\,)$ and $K(\sqrt{\beta f}\,)$ where
$\beta$ is a non-square in $\F_q^\times$.
\end{rem}

We now proceed to the ramified case.
\begin{eqnarray*}
F_R(u) &=&\frac{u^{\deg v_0}}{1+u^{\deg v_0}}
A(u)+\frac{(-u)^{\deg v_0}}{1+(-u)^{\deg v_0}}
A(-u)\\
&=&(1-q u^2)\left(\frac{u^{\deg v_0}(1+u)}{(1+u^{\deg v_0})(1-qu)}
+\frac{(-u)^{\deg v_0}(1-u)}{(1+(-u)^{\deg v_0})(1+qu)}\right).
\end{eqnarray*}

We have
\begin{eqnarray*}
 A(u)&=&(1-q u^2)\frac{1+u}{1-qu}\\
&=&1+(q+1)u+q^2u^2+\sum_{n=3}^\infty (q^n-q^{n-2})u^n.
\end{eqnarray*}

Thus,
\begin{eqnarray*}
 \frac{u^{\deg v_0}}{1+u^{\deg v_0}}A(u)&=& \left(\sum_{k=1}^\infty (-1)^{k-1}u^{k\deg v_0}\right)\left(1+(q+1)u+q^2u^2+\sum_{n=3}^\infty (q^n-q^{n-2})u^n\right)\\
&=&\sum_{k=1}^\infty (-1)^{k-1}u^{k\deg v_0}+(q+1)\sum_{k=1}^\infty (-1)^{k-1}u^{k\deg v_0+1}+q^2\sum_{k=1}^\infty (-1)^{k-1}u^{k\deg v_0+2}\\
&&+\sum_{k=1}^\infty \sum_{n=3}^\infty (-1)^{k-1}  (q^n-q^{n-2})u^{k\deg v_0+n}\\
&=&\sum_{k=1}^\infty (-1)^{k-1}u^{k\deg v_0}+(q+1)\sum_{k=1}^\infty (-1)^{k-1}u^{k\deg v_0+1}+q^2\sum_{k=1}^\infty (-1)^{k-1}u^{k\deg v_0+2}\\
&&+\sum_{m=3+\deg v_0}^\infty \sum_{k=1}^{\left\lfloor\frac{m-3}{\deg v_0}\right\rfloor} (-1)^{k-1}  (q^{m-k\deg v_0}-q^{m-k\deg v_0-2})u^{m}\\
&=&\sum_{m=3+\deg v_0}^\infty \frac{1-q^{-2}}{1+q^{-\deg v_0}} q^{m-\deg v_0}u^{m}+O_q(1)\sum_{m=\deg v_0}^\infty u^m.\\
\end{eqnarray*}

By identifying the coefficients of $F_R(u)$ with the power series
\[\sum_{\od=1}^\infty N(\Z/2\Z,\od, v_0, \mbox{ramified}) u^\od,\]
we obtain,
\[N(\Z/2\Z, \od, v_0, \mbox{ramified})=\frac{(1-q^{-2})}{1+q^{-\deg v_0}} q^{\od-\deg v_0}+O_q(1).\]

 \section{Distribution of the number of points on covers}

 \label{section-equivalence}

 We explain in this section how the results of this paper apply to the
distribution for the number of $\F_{q}$-points on covers $C$ on the moduli
space $\mathcal H_{g, \ell}$. We prove Theorem \ref{thm-HG} and make a comparison with the results of \cite{bdfl}.

 Consider an $\ell$-cyclic cover $C \to \PP^1$ defined over $\F_q$ and let $L$ be the function field of $C.$ As mentioned in Section \ref{sec:coverext} the genus $g_C$ of the cover $C$ is related to the discriminant  $\disc (L/K)$ via
\[2g_C =   (\ell -1) \left[-2 + \deg \cond(L/K)\right],\]    which implies
 \begin{eqnarray} \label{genus-conductor}
 \od =  \frac{2g_C}{\ell-1} + 2,    \end{eqnarray}
 where $\od$ is the degree of $\cond(L/K).$

Recall that the zeta function of a curve $C$ is given by
\begin{equation} \label{ZC}
Z_C(u) = \exp \left( \sum_{n=1}^\infty \# C(\F_{q^n}) \frac{u^n}{n} \right). \end{equation}

Moreover, \[Z_L(u) = Z_C(u)\] with the usual identification $u = q^{-s}.$

We recall that $\mathcal V_K$ is the set of places of $K$. Suppose that $L/K$ is a
Galois extension. We can write
\begin{eqnarray} \label{ZL}
Z_L(u) = \prod_{v \in \mathcal V_K} \left( 1 - u^{f(v) \deg{v}} \right)^{-r(v)},
\end{eqnarray}
where for each place $v$, we denote by $e(v), f(v), r(v)$ the ramification
degree, the inertia  degree and the number of places of $L$ above $v$
respectively.

Taking the logarithm on both sides of the equality $Z_C(u) = Z_L(u)$ using \eqref{ZC} and \eqref{ZL},
we obtain
\begin{eqnarray*}
\sum_{n=0}^\infty \# C(\F_{q^n}) \frac{u^n}{n} = \sum_{v \in \mathcal V_K}
\sum_{m=1}^{\infty} r(v) \frac{u^{m f(v) \deg{v}}}{m}.
\end{eqnarray*}

Equating the coefficients of $u^n$ on both sides gives
\begin{eqnarray} \label{number-points}
\#C(\F_{q^n}) = \sum_{\substack{v\in \mathcal V_K \\ f(v)  \deg{v} \mid n}}
r(v) f(v)  \deg{v} .\end{eqnarray}

The above discussion implies that the fiber above an $\F_q$-point of $\PP^1$ that corresponds to the place $v$ of degree $1$ of $K$  contains
\[\begin{cases} \ell \textrm { distinct $\F_q$-points}  & \textrm{ if $v$
splits completely,} \\
1 \textrm { $\F_q$-point} & \textrm{ if $v$ ramifies,}\\
0 \textrm{ $\F_q$-points} & \textrm{ if $v$ is inert}.
\end{cases}\]

More generally, a place $v$ of $K$  corresponds to a Galois orbit
of rational points of the same degree of $\PP^1.$ The fiber above each point in the orbit contains
\[\begin{cases} \ell \textrm { distinct points of degree}\, \deg v   & \textrm{ if $v$
splits completely,} \\
1 \textrm { point of degree  } \deg v & \textrm{ if $v$ ramifies,}\\
1 \textrm{ point of degree } \ell \deg v & \textrm{ if $v$ is inert}.
\end{cases}\]

To get the distribution of $\# C(\F_{q})$ over $\mathcal H_{g,\ell}$, we use the relative densities
$$
\frac{ N(\Z / \ell\Z, \frac{2g}{\ell-1} + 2; \mathcal{V}_R, \mathcal{V}_S, \mathcal{V}_I)}{N(\Z/ \ell\Z,\frac{2g}{\ell-1} + 2)}
$$
where we take the sets $\mathcal{V}_R, \mathcal{V}_S, \mathcal{V}_I$ to be mutually disjoint, and such that
$\mathcal{V}_R\cup \mathcal{V}_S\cup \mathcal{V}_I$ is a subset of the set of places of degree 1 in $\CS_K$.

Then,
using \eqref{number-points} with $n=1$ and Theorem \ref{thm:main}, we get
\begin{eqnarray*}
&&\frac{ \left| \left\{ C \in \mathcal H_{g, \ell} (\F_q) \;:\; \#C(\F_q) = m \right\}\right| }{\left| \mathcal H_{g, \ell}(\F_q)\right|} \\ && =
\sum_{\ell|\mathcal{V}_S|+|\mathcal{V}_R| = m} \frac{ N(\Z / \ell\Z, \frac{2g}{\ell-1} + 2; \mathcal{V}_R, \mathcal{V}_S, \mathcal{V}_I)}{N(\Z/ \ell\Z,\frac{2g}{\ell-1} + 2)}\\
&\sim& \sum_{\ell|\mathcal{V}_S|+|\mathcal{V}_R| = m} \left( \frac{\ell-1}{q+\ell-1} \right)^{|\mathcal{V}_R|} \left( \frac{q}{\ell(q + \ell-1)} \right)^{|\mathcal{V}_S|}
 \left( \frac{(\ell-1)q}{\ell(q + \ell-1)} \right)^{q+1-|\mathcal{V}_R|} \\
&=&  \mbox{Prob} \left( \sum_{i=1}^{q+1} X_i = m \right),
\end{eqnarray*}
where the $X_i$ are the random variables of Theorem \ref{thm-HG}.

\subsection{Affine models}\label{subsec:affine}
We compare the results of this paper with the results of \cite{bdfl} concerning the irreducible components $\mathcal H^{(d_1,
\dots, d_\ell)}$ of $\mathcal H_{g, \ell}$.
To describe these components,  we write the covers concretely in terms of affine
models. Each such cover has an affine model of the form
\begin{equation}\label{eq:affinemodel}
C: Y^\ell = f(X) =  \beta f_1 f_2^2 \cdots f_{\ell-1}^{\ell-1}
\end{equation}
where the $f_i \in \F_q[X]$ are monic, square-free, pairwise coprime, of
degrees $d_1, \dots, d_{\ell-1}.$ The degree of the conductor depends on the
degrees $d_1, \dots, d_{\ell -1}$ and whether there is ramification at the
place at infinity.  The ramification at the place at infinity is determined by whether
the total degree of the polynomial is divisible by $\ell$.  When $d_1+\cdots
+(\ell-1) d_{\ell-1}$ is a multiple of $\ell$, then the cover does not ramify at
infinity, otherwise there is ramification at infinity.  In the first case the
degree of the conductor is  $d_1+\cdots +d_{\ell-1}$ and in the second case it
is  $d_1+\cdots +d_{\ell-1} + 1$.

By the Riemann--Hurwitz formula the genus of this cover is given by
\[g_C=(\ell-1)(d_1+\cdots +d_{\ell-1}-2)/2\]
in the first case, and by
\[g_C=(\ell-1)(d_1+\cdots +d_{\ell-1}-1)/2,\]
in the second. Both equations are compatible with the relation \eqref{genus-conductor} between
the genus $g_C$ and the degree of the conductor $\od$.

For a given conductor, each $\beta \in \F_q^\times/(\F_q^\times)^\ell$ yields a different cover given by formula \eqref{eq:affinemodel}.
That is, there is one such extension
for each element of $\F_q^\times/(\F_q^\times)^\ell$. Using the notation from  \cite{bdfl}, we define
\[\mathcal{F}_{(d_1,\dots, d_{\ell-1})} =\{(f_1,\dots, f_{\ell-1}) :  f_i \text{
monic, square-free, pairwise coprime, } \deg f_i=d_i, i=1, \dots, \ell-1 \}.\]

We consider, for $d_1 + \cdots + (\ell-1) d_{\ell-1} \equiv 0 \pmod \ell$,
\[{\mathcal{F}}_{[d_1,\dots, d_{\ell-1}]}
  = \mathcal{F}_{(d_1,\dots, d_{\ell-1})} \cup \bigcup_{j=1}^{\ell-1}
\mathcal{F}_{(d_1,\dots, d_j-1,\dots, d_{\ell-1})}.\]
The elements in the first term correspond to affine models from Equation \eqref{eq:affinemodel} in the case in which there
is no ramification at infinity. The elements in the other terms correspond to affine models where there is ramification at infinity
(of index $j$ in the term $\mathcal{F}_{(d_1,\dots, d_j-1,\dots, d_{\ell-1})}$).

Now suppose that we want to count the number of covers of genus $g$. For a
conductor $f_1f_2^2\cdots f_{\ell-1}^{\ell-1}$, there are $\ell$
different covers according to
the class of the leading coefficient as an element of
$\F_q^\times/(\F_q^\times)^\ell$. In addition, we need to consider that isomorphic covers are obtained by taking
automorphisms of $\mathbb{P}^1(\F_q)$, namely the $q(q^2-1)$ elements of $\mathrm{PGL}_2(\F_q)$ (see Section 7 of \cite{bdfl} for details).

We denote
by $\mathcal{H}^{(d_1,\dots,d_{\ell})}$ the corresponding component of $\mathcal{H}_g$, indexed by tuples $(d_1,\dots,d_{\ell})$
with the properties that $(\ell-1)(d_1+\cdots+d_{\ell-1}-2)=2g$ and $d_1+2d_2+\cdots +(\ell-1)d_{\ell-1}\equiv 0 \,(\mathrm{mod}\, \ell)$. The discussion in the previous paragraphs implies
\[|\mathcal{H}^{(d_1, \dots,d_{\ell-1})}|'=\frac{\ell}{q(q^2-1)}|{\mathcal{F}}_{[d_1,\dots, d_{\ell-1}]}|\]
where, as usual, the $'$ notation means that the covers $C$ on the moduli space are counted with the usual weights $1/|\mathrm{Aut}(C)|$.

Thus, we can write
\begin{equation}\label{sumL}
|\mathcal{H}_{g,\ell}(\F_q)|'
=\frac{\ell}{q(q^2-1)}\sum_{\substack{{d_1+\cdots+d_{\ell-1}=2(g+2)/(\ell-1)}\\{d_1+\cdots
+(\ell-1)d_{\ell-1}\equiv 0\!\! \pmod \ell}}}|\mathcal{F}_{[d_1,\dots,
d_{\ell-1}]}|.
\end{equation}

Formula (3.1) of \cite{bdfl} says that
\begin{eqnarray} \label{fromBDFL} |\mathcal{F}_{(d_1,  \dots,
d_{\ell-1})}|=\frac{L_{\ell-2}q^{d_1+\cdots+d_{\ell-1}}}{\zeta_q(2)^{\ell-1}}
\left(1+ O\left(\sum_{h=2}^{\ell-1}q^{\varepsilon(d_h+\cdots +d_{\ell-1})-d_h}
+q^{-d_1/2 }  \right)\right),\end{eqnarray}
where
\[L_{\ell-2}=\prod_{j=1}^{\ell-2} \prod_{v \not = v_\infty}
\left(1-\frac{jq^{-2\deg v}}{(1+q^{-\deg v})(1+jq^{-\deg v})}\right).\]

The formula above may be rewritten as
\begin{align*}|\mathcal{F}_{(d_1,  \dots,
d_{\ell-1})}|
&=\frac{(\ell-2)!C_\ell q^{d_1+\cdots+d_{\ell-1}}}{
(1+(\ell-1)q^{-1})} \left(1+ O\left(\sum_{h=2}^{\ell-1}q^{\varepsilon(d_h+\cdots
+d_{\ell-1})-d_h} +q^{-d_1/2 }  \right)\right).\\
 \end{align*}

This gives
\begin{eqnarray*}
q(q^2-1)|\mathcal{H}_{g,\ell}(\F_q)|'&=&(\ell-2)!C_\ell q^{\od}
\sum_{\substack{{d_1+\cdots+d_{\ell-1}=\od}\\{d_1+\cdots +(\ell-1)d_{\ell-1}\equiv 0
\!\pmod  \ell}}}\ell  +ET\\
&=&C_\ell \od^{\ell-2} q^{\od}
+ET,
\end{eqnarray*}
where $ET$ denotes an error term and in the last line, we used that the number of solutions of $d_1+\cdots
+d_{\ell-1}=\od$
is given by $\binom{\od+\ell-2}{\od}\sim \frac{\od^{\ell-2}}{(\ell-2)!}$.

Thus the result of Theorem \ref{caseL} is compatible with the result of
Theorem 3.1 from \cite{bdfl}, in the sense that  summing the main terms of
\eqref{fromBDFL} gives the same number of elements of $\mathcal H_{g, \ell}(\F_q)$ computed in this paper,
even if the error terms coming from \eqref{fromBDFL} are only valid when  $\min\{d_1, \dots, d_{\ell-1}\} \rightarrow \infty$.

\section{Acknowledgments} This work was initiated at a workshop  at the American Institute for Mathematics (AIM)  and a substantial part of this work was completed during a SQuaREs program at AIM. The authors would like to thank AIM for
these opportunities. The authors would also like to thank Colin Weir and Simon Rubinstein-Salzedo for useful discussions.

This work was supported by
  the Simons Foundation [\#244988 to A. B.], the National Science Foundation [DMS-1201446 to B. F., DMS-1147782 and DMS-1301690 to M. M. W. ], PSC-CUNY [to B. F.], the Natural Sciences and Engineering Research Council
 of Canada [Discovery Grant 155635-2013 to C. D., 355412-2013 to M. L.], the Fonds de recherche du Qu\'ebec - Nature et technologies [144987 to M. L., 166534 to C. D. and M. L.],  the Scientific and Research Council of Turkey [TUBITAK-114C126 to E. O.] and the American Institute for Mathematics [Five Year Fellowship to M. M. W.].


\begin{thebibliography}{CDyDO02}

\bibitem[BDFL10]{bdfl}
Alina {Bucur}, Chantal {David}, Brooke {Feigon}, and Matilde {Lal{\'\i}n}.
\newblock {Statistics for traces of cyclic trigonal curves over finite fields.}
\newblock {\em {Int. Math. Res. Not.}}, 2010(5):932--967, 2010.

\bibitem[CDyDO02]{CDO}
Henri Cohen, Francisco Diaz~y Diaz, and Michel Olivier.
\newblock On the density of discriminants of cyclic extensions of prime degree.
\newblock {\em J. Reine Angew. Math.}, 550:169--209, 2002.

\bibitem[Coh54]{Co54}
Harvey Cohn.
\newblock The density of abelian cubic fields.
\newblock {\em Proc. Amer. Math. Soc.}, 5:476--477, 1954.

\bibitem[CWZar]{ChWoZa2012}
GilYoung Cheong, Melanie~Matchett Wood, and Azeem Zaman.
\newblock The distribution of points on superelliptic curves over finite
  fields.
\newblock {\em Proceedings of the American Mathematical Society}, to appear.
\newblock Available at arXiv:1210.0456.

\bibitem[Hay74]{hayes}
D.~R. Hayes.
\newblock Explicit class field theory for rational function fields.
\newblock {\em Trans. Amer. Math. Soc.}, 189:77--91, 1974.

\bibitem[KR09]{KuRu2009}
P{\"a}r Kurlberg and Ze{\'e}v Rudnick.
\newblock The fluctuations in the number of points on a hyperelliptic curve
  over a finite field.
\newblock {\em J. Number Theory}, 129(3):580--587, 2009.

\bibitem[Mor91]{moreno}
Carlos Moreno.
\newblock {\em Algebraic curves over finite fields}, volume~97 of {\em
  Cambridge Tracts in Mathematics}.
\newblock Cambridge University Press, Cambridge, 1991.

\bibitem[Ros02]{rosen}
Michael Rosen.
\newblock {\em Number theory in function fields}, volume 210 of {\em Graduate
  Texts in Mathematics}.
\newblock Springer-Verlag, New York, 2002.

\bibitem[Rud10]{Ru-hyper}
Ze{\'e}v Rudnick.
\newblock Traces of high powers of the {F}robenius class in the hyperelliptic
  ensemble.
\newblock {\em Acta Arith.}, 143(1):81--99, 2010.

\bibitem[TX14]{ThXI}
Frank Thorne and Maosheng Xiong.
\newblock Distribution of zeta zeroes for cyclic trigonal curves over a finite
  field.
\newblock 2014.
\newblock (preprint).

\bibitem[VS06]{Sal}
Gabriel~Daniel Villa~Salvador.
\newblock {\em Topics in the theory of algebraic function fields}.
\newblock Mathematics: Theory \& Applications. Birkh\"auser Boston, Inc.,
  Boston, MA, 2006.

\bibitem[{Woo}10]{woodabelian}
Melanie~Matchett {Wood}.
\newblock {On the probabilities of local behaviors in abelian field
  extensions.}
\newblock {\em {Compos. Math.}}, 146(1):102--128, 2010.

\bibitem[Woo12]{Wo2012}
Melanie~Matchett Wood.
\newblock The distribution of the number of points on trigonal curves over
  {$\Bbb F_q$}.
\newblock {\em Int. Math. Res. Not. IMRN}, (23):5444--5456, 2012.

\bibitem[Wri89]{Wr1989}
David~J. Wright.
\newblock Distribution of discriminants of abelian extensions.
\newblock {\em Proc. London Math. Soc. (3)}, 58(1):17--50, 1989.

\bibitem[Xio10]{Xi2010}
Maosheng Xiong.
\newblock The fluctuations in the number of points on a family of curves over a
  finite field.
\newblock {\em J. Th\'eor. Nombres Bordeaux}, 22(3):755--769, 2010.

\end{thebibliography}
\end{document}